\documentclass[a4paper,12pt]{amsart}
\usepackage{graphicx}
\usepackage{amssymb}
\usepackage{amsmath}
\usepackage{latexsym}
\usepackage{amsthm}
\usepackage{autograph,epic,latexsym,bezier,amsbsy,color,enumerate,amsfonts,amsmath,amscd,amssymb}
\usepackage[latin1]{inputenc}

 \setloopdiam{17}\setprofcurve{10}

\newtheorem{thm}{Theorem}
\newtheorem{prop}[thm]{Proposition}

\newtheorem{lemma}{Lemma}
\theoremstyle{definition}
\newtheorem{defi}{Definition}
\newtheorem{example}{Example}
\newtheorem{remark}{Remark}

\begin{document}

\title[Lumpability for Markov chains on poset block structures]{The lumpability property for a family of Markov chains on poset block structures}

\author{Daniele D'Angeli}
\address{Departamento de Matem\'{a}tica, Universidade de Tr\'{a}s-os-Montes e Alto Douro, Quinta do Prado, Vila Real 5001-801, Portugal}
\email{dangeli@utad.pt}

\author{Alfredo Donno}
\address{Dipartimento di Scienze di Base e Applicate per l'Ingegneria, Sezione di Matematica, Sapienza Università di Roma,
Via A. Scarpa, 16 00161 Roma, Italia. Phone: +39 06 49766627.}
\email{alfredo.donno@sbai.uniroma1.it}


\date{\today, preprint}

\begin{abstract}
We construct different classes of lumpings for a family of Markov
chain products which reflect the structure of a given finite
poset. We use essentially combinatorial methods. We prove that,
for such a product, every lumping can be obtained from the action
of a suitable subgroup of the generalized wreath product of
symmetric groups, acting on the underlying poset block structure,
if and only if the poset defining the Markov process is totally
ordered, and one takes the uniform Markov operator in each factor
state space. Finally we show that, when the state space is a
homogeneous space associated with a Gelfand pair, the spectral
analysis of the corresponding lumped Markov chain is completely
determined by the decomposition of the group action into
irreducible submodules.
\end{abstract}

\maketitle

\begin{center}
\end{center}

\section{Introduction}

The study of Markov chains is one of the most useful tools of
Probability and has several fundamental applications in many areas
of science: Mathematics, Statistics, Physics, Chemistry,
Mathematical Biology, Information Science, but also in Economics
and Social Sciences, due to the fact that Markov chains can be
used for modelling an infinitely large
variety of evolution processes.\\
\indent One of the most interesting and studied problems in Markov
chain Theory is the investigation of the rate of convergence to
the stationary distribution, which is mainly performed via the
spectral analysis of the associated Markov operator. One is often
interested in establishing if the convergence to the stationary
distribution presents a cutoff. The term cutoff is used to
describe the situation where the total variation distance between
the $k$-step transition probability and the stationary
distribution stays close to its maximum value at $1$ for a while,
then suddenly drops to a quite small value and tends to $0$
exponentially fast \cite{cut1, diaconisbook}. The presence of such
a phenomenon was proven in a number of models \cite{cut2, cut3,
cut4}; on the other hand, it was shown in \cite{cutoff} that the
Markov chain described in Section \ref{cutoffcitation} of the
present paper converges to the limit distribution without a cutoff.\\
\indent Lumping a Markov chain appears as a useful tool in this
kind of investigation, since by lumping a Markov chain the
spectral gap cannot decrease. Informally, when a Markov chain is
lumpable, it is possible to reduce the number of states by a sort
of aggregation process, obtaining a \lq\lq smaller\rq\rq Markov
chain. Notice that in \cite[Chapter 1]{woess} the lumping
construction is called factorization, and the lumped Markov chain
is called the factor chain. In \cite{franceschinis}, the weaker
notion of quasi-lumpability is introduced: roughly speaking, a
quasi-lumpable Markov chain is a Markov chain that is not
lumpable, but can be altered by a small perturbation in such a way
that the resulting Markov chain is lumpable in the classical
sense. In \cite{farago}, the author provides some new bounds on
the rate of convergence of aggregated Markov chains; in
\cite{farago2}, the lumpability and quasi-lumpability properties
are used in order to speed up the Markov Chain Monte Carlo, i.e.,
to reduce its running time necessary to get a sufficiently good
approximation of the stationary distribution. In \cite{zhoulange},
the lumpability property is used in the study of a family of
composition Markov chains, having numerous applications as, for
instance, to Ehrenfest chains, or to Kimura's continuous time
chain for the DNA evolution.\\
\indent In our paper, we consider a family of finite Markov
chains, that we call generalized crested products, introduced in
\cite{generalizedcrested} and providing a generalization of the
crested product of Markov chains previously defined in
\cite{crested}, motivated by the construction developed in
\cite{treesarticle} in the context of Gelfand pairs. The theory of
crested products was introduced by R. A. Bailey and P. J. Cameron
\cite{baileycrested} in the setting of symmetric association
schemes and groups, and then extended to the case of arbitrary
permutation representations by F. Scarabotti and F. Tolli in
\cite{plms}. Observe that the crested product of Markov chains
also generalizes some classical diffusion models, e.g. the
Ehrenfest model and the Laplace-Bernoulli model (see also
\cite{symmetrical}, where the spectral analysis of a combination
of these models is performed by using Gelfand pair theory). Our
Markov chains are naturally defined starting from a finite poset
$(I,\preceq)$: given a Markov chain with finite space state $X_i$,
for each $i\in I$, the generalized crested product is a Markov
chain reflecting the combinatorial structure of the poset
$(I,\preceq)$, and whose state space is the cartesian product
$X_1\times \cdots \times
X_n$.\\
\indent After recalling the definition and some basic facts about
lumpable Markov chains, generalized crested products and
generalized wreath products (Section \ref{sectionpreliminaries}),
we construct in Section \ref{sectionlumpings} different classes of
lumpings for the generalized crested products of Markov chains:
the deletion construction (see Theorem \ref{forgetisgood} and
Propositions \ref{propositionR} and \ref{deletingreduced}),
obtained by identifying elements which differ only for coordinates
indexed by a subset $R$ of $I$; the direct product of lumpings
(see Theorem \ref{thm6}), where a lumping partition is obtained as
a direct product of lumpings of each factor $X_i$; and the
generalized crested product of lumpings (see Theorem \ref{thm7}),
which is a partition equivalent of the construction introduced in
\cite{bayleygeneralized} for permutation groups, taking into
account the combinatorial structure of the poset $(I,\preceq)$.\\
\indent In Section \ref{labelinsect}, we recall the so called
Insect Markov chain defined in \cite{orthogonal} in the more
general setting of orthogonal block structures, and generalizing a
model introduced in \cite{figatalamanca}, whose state space can be
identified with the boundary of a combinatorial structure called
poset block structure, whose automorphism group is the generalized
wreath product of permutation groups defined in
\cite{bayleygeneralized}, and which can be endowed with a metric
space structure (Lemma \ref{lemmametricspace}). Finally, in
Section \ref{sectiongroups}, we focus our attention on lumped
Markov chains that can be obtained by using the orbit partition
associated with the action of an automorphism group on the state
space. In particular, we prove that if the poset $(I,\preceq)$ is
a totally ordered set, so that the corresponding poset block
structure is a rooted tree, then every lumping of the Insect
Markov chain can be obtained from the action of a suitable
automorphism group (see Theorem \ref{teoinsetto}). We show that,
when the state space can be expressed as a homogeneous space with
respect to the stabilizer of some fixed element, then the Gelfand
pair theory developed in \cite{orthogonal} allows to give a
complete spectral analysis. On the other hand, if the poset
$(I,\preceq)$ is not a chain, we prove that the Insect Markov
chain admits a lumping that cannot be obtained using the action of
an automorphism group (see Proposition \ref{prop14}).

\section{Preliminaries}\label{sectionpreliminaries}
\subsection{Lumpable Markov chains}

We recall in this section some basic facts about finite Markov
chains and lumpability property (see, for instance, \cite[Chapter
1]{libro}).

Let $X$ be a finite set and let $P=(p(x,y))_{x,y\in X}$ be a
stochastic matrix indexed by the elements of $X$. Consider a
Markov chain on $X$ with transition probability matrix $P$. By
abuse of notation, we will denote by $P$ the Markov chain as well
as the associated Markov operator on $L(X) = \{f:X \longrightarrow
\mathbb{C}\}$, defined as
$$
(Pf)(x) = \sum_{y\in X}p(x,y)f(y), \qquad \text{for all } f \in
L(X), x\in X.
$$
\begin{defi}
The Markov chain $P$ is reversible if there exists a strict
probability measure $\pi$ on $X$ such that
$$
\pi(x) p(x,y) = \pi(y) p(y,x), \qquad \text{for all } x,y \in X.
$$
\end{defi}
If this is the case, we say that $P$ and $\pi$ are in detailed
balance \cite{aldous}. Define a scalar product on $L(X)$ as
$$
\langle f_1,f_2 \rangle_{\pi} = \sum_{x\in X}f_1(x)
\overline{f_2(x)}\pi(x), \qquad \text{for all } f_1, f_2\in L(X).
$$
It is easy to verify that $\pi$ and $P$ are in detailed balance if
and only if $P$ is self-adjoint with respect to the scalar product
$\langle \cdot, \cdot \rangle_{\pi}$. Under these hypotheses, it
is known that the matrix $P$ can be diagonalized over
$\mathbb{R}$. Let $\sigma(P)$ denote the spectrum of $P$. Then,
one has $1\in \sigma(P)$ and $|\lambda|\leq 1$ for any $\lambda\in
\sigma(P)$.

Starting from a Markov chain $P$ with state space $X$, the notion
of lumpability allows to construct a new Markov chain which has a
smaller state space. So let $\mathcal{L} = \{L_1,\ldots,L_k\}$ be
a partition of $X$, i.e., $X= \sqcup_{i=1}^k L_i$. Then, for each
$x\in X$ and $i=1,\ldots, k$, put
$$
p(x, L_i) = \sum_{y\in L_i} p(x,y).
$$
\begin{defi}
The Markov chain $P$ is {\it lumpable} with respect to the
partition $\mathcal{L}$ if, for any pair $L_i, L_j\in
\mathcal{L}$, the function $x\mapsto p(x,L_j)$ is constant on
$L_i$.
\end{defi}
If this is the case, we put $\widetilde{p}(L_i,L_j)=p(x,L_j)$, for
any $L_i,L_j\in \mathcal{L}$ and $x\in L_i$. Therefore, a
stochastic matrix $\widetilde{P} =
(\widetilde{p}(L_i,L_j))_{i,j=1,\ldots, k}$ can be defined, so
that $\widetilde{P}$ can be regarded as the transition probability
matrix of the lumped Markov chain, whose state space contains $k$
elements, identified with the parts of the partition
$\mathcal{L}$. We will say that $\mathcal{L}$ is a
\textit{lumping} of $P$. It is easy to check that, if $P$ is in
detailed balance with respect to the probability measure $\pi$,
then $\widetilde{P}$ is reversible with respect to
$\pi|_{\mathcal{L}}$, where $\pi(L_i)$ is defined by $\pi(L_i)=
\sum_{x\in L_i}\pi(x)$.\\
\indent Roughly speaking, lumpability means that some states of
$P$ can be aggregated and replaced by a single state, providing a
Markov chain which has a smaller state space, but whose behavior
is essentially the same as the original one. Observe that there
can exist many different lumpings of the same Markov chain, which
are not necessarily obtained via successively refinements of the
corresponding partitions (see, for instance, Example
\ref{differentdeletions}).

Let us denote by ${\bf 1}_{L_i}$, for $i=1,\ldots, k$, the
characteristic function of $L_i$, i.e,
$$
{\bf 1}_{L_i}(x)=\begin{cases}
1&\text{if}\ x\in L_i\\
0& \text{otherwise}
\end{cases}
$$
and let $W_{\mathcal{L}}$ be the subspace of $L(X)$ generated by
the functions $\{{\bf 1}_{L_i}\}_{i=1,\ldots, k}$. In other words,
$W_{\mathcal{L}}$ is the subspace of $L(X)$ constituted by the
functions which are constant on each part of $\mathcal{L}$. Denote
by $\delta_{L_i}$ the Dirac function centered at $L_i$ in
$L(\mathcal{L})$.

\begin{prop}\label{lumpableproperties}
$P$ is lumpable with respect to $\mathcal{L}$ if and only if the
subspace $W_{\mathcal{L}}$ is $P$-invariant. Let $x\in X$ and
suppose that $x\in L_j$, for some $j\in \{1,\ldots,k\}$. Then
$P({\bf 1}_{L_i})(x) = \widetilde{P}(\delta_{L_i})(L_j)$, so that
there exists a bijection between the eigenfunctions of the lumped
chain and the eigenfunctions of $P$ which are constant on each
part of $\mathcal{L}$. If $f$ is such an eigenfunction, with
associated eigenvalue $\lambda$, then the lumped eigenfunction
$\widetilde{f}$ defined by $\widetilde{f}(L_j) = f(x)$ is well
defined and has eigenvalue $\lambda$.
\end{prop}

\begin{proof}
Let $x\in L_j$. We have
$$
(P{\bf 1}_{L_i})(x) = p(x,L_i) = \sum_{y\in L_i}p(x,y).
$$
and therefore $P{\bf 1}_{L_i}$ is constant on each part $L_j$ of
$\mathcal{L}$ if and only if $p(x,L_i)$ does not depend on $x\in
L_j$. If this is the case, one has
$p(x,L_i)=\widetilde{p}(L_j,L_i)=\widetilde{P}(\delta_{L_i})(L_j)$.
Now let $f$ be an eigenfunction of $P$ with eigenvalue $\lambda$,
which is constant on each part of $\mathcal{L}$. One has:
$$
\lambda\widetilde{f}(L_j)= \lambda f(x) = (Pf)(x) =
\sum_{i=1}^k\!\left(\sum_{y\in L_i}p(x,y)f(y)\!\right)=
\sum_{i=1}^k\widetilde{p}(L_j,L_i)\widetilde{f}(L_i) =
(\widetilde{P}\widetilde{f})(L_j),
$$
so that $\widetilde{f}$ still has eigenvalue $\lambda$.
\end{proof}
It follows from Proposition \ref{lumpableproperties} that
$\sigma(\widetilde{P})\subseteq \sigma(P)$. Thus, the spectral gap
for $\widetilde{P}$ is never smaller than that of $P$. This fact
is what makes lumping a possible strategy for accelerating
convergence to the stationary distribution \cite{addedreference}.

\begin{example}\rm
Consider the Markov chain on the set $X=\{x_0,x_1,x_2,x_3\}$, with
associated transition probability matrix
$P=(p(x_i,x_j))_{i,j=0,1,2,3}$ given by
$$
P = \frac{1}{12}\begin{pmatrix}
  5& 5 & 1 & 1 \\
  5 & 5 & 1 & 1 \\
  1 & 1 & 5 & 5 \\
  1 & 1 & 5 & 5
\end{pmatrix}.
$$
Take the partition $\mathcal{L}$ of $X$ given by $X=L_1\sqcup
L_2$, with $L_1=\{x_0,x_2\}$, $L_2=\{x_1,x_3\}$. $P$ is lumpable
with respect to $\mathcal{L}$ and the matrix $\widetilde{P}$ of
the lumped Markov chain is given by
$$
\widetilde{P} = \frac{1}{2}\begin{pmatrix}
  1 & 1 \\
  1 & 1
\end{pmatrix}.
$$
We have $\sigma(P) = \{1, \frac{2}{3},0\}$ and the associated
eigenspaces are $W_1$, generated by the function $f_1=(1,1,1,1)$,
$W_{2/3}$, generated by the function $f_{2/3}=(1,1,-1,-1)$ and
$W_0$, generated by the functions $f_{0,1}=(1,-1,0,0)$ and
$f_{0,2}=(0,0,1,-1)$. Moreover, $\sigma(\widetilde{P}) = \{1,0\}$
and the associated eigenspaces are $\widetilde{W}_1$, generated by
the function $\widetilde{f}_1=(1,1)$ and $\widetilde{W}_0$,
generated by the function $\widetilde{f}_0=(1,-1)$. Note that the
eigenspace $W_{2/3}$ does not contain any function which is
constant on $L_1$ and $L_2$.
\end{example}

Now suppose that $P$ is lumpable with respect to the partition
$\mathcal{L} = \{L_1,\ldots,L_k\}$ of $X$ and let $\widetilde{P}$
be the associated lumped Markov chain. Let $\widetilde{P}$ be
lumpable with respect to the partition $\mathcal{M}=\{M_{I_1},
\ldots, M_{I_h}\}$ of $\mathcal{L}$, where $\{1, \ldots, k \} =
\sqcup_{j=1}^hI_j$, $M_{I_j}= \sqcup_{s\in I_j}L_s$ and clearly
$h\leq k$. Let us denote $\widetilde{\widetilde{P}}$ the
associated lumped Markov chain on $\mathcal{M}$. Notice that
$\mathcal{M}$ is coarser than $\mathcal{L}$ as a partition of $X$.
We claim that $\mathcal{M}$ is also a lumping partition of $P$.

Let $x\in L_{s_i}\subseteq M_{I_i}$. We have to prove that
$p(x,M_{I_j})$ is constant on $M_{I_i}$. We have

\begin{eqnarray*}
p(x,M_{I_j})&=&\sum_{y\in M_{I_j}}p(x,y)= \sum_{s\in
I_j}\sum_{y\in
L_s}p(x,y)\\
&=& \sum_{s\in I_j}\widetilde{p}(L_{s_i},L_s) =
\widetilde{\widetilde{p}}(M_{I_i},M_{I_j}),
\end{eqnarray*}
which does not depend on the particular choice of $x\in M_{I_i}$.

\subsection{Generalized crested product of Markov chains}

Let $(I,\preceq)$ be a finite poset. For every $i\in I$, and
$J\subseteq I$, the following subsets of $I$ can be defined
\cite{bayleygeneralized}:
\begin{itemize}
\item $A(i)=\{j\in I : j \succ i\}$, $A[i] = A(i) \sqcup \{i\}$,
$A(J)=\cup_{j\in J}A(j)$, $A[J]= \cup_{j\in J}A[j]$.
\item $H(i)=\{j\in I : j \prec i\}$, $H[i] = H(i) \sqcup \{i\}$, $H(J)=\cup_{j\in J}H(j)$, $H[J]= \cup_{j\in J}H[j]$.
\end{itemize}
A subset $J\subseteq I$ is said {\it ancestral} if, whenever $i
\succ j$ and $j\in J$, then $i\in J$. Note that by definition
$A(i)$ and $A[i]$ are ancestral, for each $i\in I$. The set $A(i)$
is called the ancestral set of $i$, whereas the set $H(i)$ is
called the hereditary set of $i$. Finally, we recall that an
antichain is a subset $S\subseteq I$ in which no two distinct
elements are comparable.

For each $i\in I$, let $X_i$ be a finite set, with $|X_i|\geq 2$,
so that we can identify $X_i$ with the set $\{0,1, \ldots,
|X_i|-1\}$. Let $P_i$ be a Markov chain on $X_i$ (as usual, we
also denote by $P_i$ the associated Markov operator on $L(X_i)$).
Let $I_i=(\delta_i(x_i,y_i))_{x_i,y_i\in X_i}$ be the identity
matrix of size $|X_i|$, and let $U_i=(u_i(x_i,y_i))_{x_i,y_i\in
X_i}$ be the matrix whose entries are all equal to $1/|X_i|$. We
still denote by $I_i$ and $U_i$ the associated Markov operators on
$L(X_i)$, that we call the identity and the uniform operator,
respectively. The generalized crested product is a new Markov
chain defined on the space $X_1\times \cdots \times X_n$.

\begin{defi}[\cite{generalizedcrested}]\label{defgencrestedproduct}
Let $(I,\preceq)$ be a finite poset, with $|I|=n$, and let
$\{p_i^0\}_{i\in I}$ be a strict probability measure on $I$, so
that $p_i^0>0$ for every $i\in I$ and $\sum_{i=1}^np_i^0=1$. The
\textit{generalized crested product} of the Markov chains $P_i$
defined by $(I,\preceq)$ and $\{p_i^0\}_{i\in I}$ is the Markov
chain on $X=X_1\times \cdots \times X_n$ whose associated Markov
operator is
\begin{eqnarray}\label{definizioneiniziale}
\mathcal{P}=\sum_{i\in I}p_i^0 \left(P_i\otimes
\left(\bigotimes_{j\in H(i)} U_j\right)\otimes
\left(\bigotimes_{j\not\in H[i]} I_j\right)\right).
\end{eqnarray}
\end{defi}
The probability transition on $X$ associated with $\mathcal{P}$
will be denoted by $p(x,y)$, for all $x,y\in X$. For $x,y \in X$
we have:
$$
p(x,y) = \sum_{i\in I}p_i^0p_i(x_i,y_i) \prod_{j\in
H(i)}u_j(x_j,y_j) \prod_{j\not \in H[i]}\delta_j(x_j,y_j).
$$
The spectral analysis of these products of Markov chains has been
performed in \cite{generalizedcrested} under the hypothesis of
irreducibility for the $P_i$'s, together with the study of
ergodicity and of the $k$-step transition probability.

\begin{remark}{\bf (The generalized Ehrenfest model) }\rm
The generalized crested product can be seen as a generalization of
the classical Ehrenfest diffusion model. This model consists of
two urns numbered $0$, $1$ and $n$ balls numbered $1, \ldots, n$.
A configuration is given by a placement of the balls into the
urns. Note that there is no ordering inside the urns. At each
step, a ball is randomly chosen (with probability $1/n$) and it is
moved to the other urn. In \cite{crested} we generalized it to the
$(C,N)$-Ehrenfest model. Now put $|X_i|=q$, for each $i=1,\ldots,
n$: then we have the following interpretation of the generalized
crested product. Suppose that we have $n$ balls numbered by $1,
\ldots, n$ and $q$ urns. Let $(I,\preceq)$ be a finite poset with
$n$ elements, so that a hierarchy is introduced in the set of
balls. At each step, we choose a ball $i$ according with a
probability distribution $p_i^0$: then we move it to another urn
following a transition probability $P_i$ and all the other balls
numbered by indices $j$ such that $j\prec i$ in the poset
$(I,\preceq)$ are moved uniformly to a new urn. The balls
corresponding to all the other indices are not moved.
\end{remark}

\subsection{Generalized wreath product of groups}\label{labelgroupsbayley}

Let $(I,\preceq)$ be a finite poset, with $|I|=n$. For each $i\in
I$, let $X_i$ be a finite set, with $|X_i|\geq 2$. For $J\subseteq
I$, put $X_J = \prod_{i\in J}X_i$. In particular, we put $X =
X_I$. If $K\subseteq J \subseteq I$, let $\pi^J_K$ denote the
natural projection from $X_J$ onto $X_K$. In particular, we set
$\pi_J = \pi^I_J$ and $x_J=\pi_J x$, for every $x\in X$. Let
$\mathcal{A}$ be the set of ancestral subsets of $I$. If $J\in
\mathcal{A}$, then the equivalence relation $\sim_J$ on $X$ is
defined as
$$
x \sim_J y \quad \Longleftrightarrow \quad x_J = y_J,\qquad
\mbox{where } x=(x_i)_{i\in I} \ \text{and } y=(y_i)_{i\in I} \in
X.
$$
\begin{defi}[\cite{bayleygeneralized}]\label{defiposetblocks}
A {\it poset block structure} is a pair $(X,\sim_{\mathcal{A}})$,
where
\begin{enumerate}
\item $X = \prod_{(I,\preceq)}X_i$, with $(I,\preceq)$ a
finite poset and $|X_i| \geq 2$, for each $i\in I$;
\item $\sim_{\mathcal{A}}$ denotes the set of equivalence relations on
$X$ defined by all the ancestral subsets of $I$.
\end{enumerate}
\end{defi}
For each $i\in I$, let $G_i$ be a permutation group on $X_i$ and
let $F_i$ be the set of all functions from $X_{A(i)}$ into $G_i$.
For $J\subseteq I$, we put $F_J = \prod_{i\in J}F_i$. An element
of $F_I$ will be denoted $f = (f_i)_{i\in I}$, with $f_i \in F_i$.

\begin{defi}\label{labelbayley}
For each $f\in F_I$, the action of $f$ on $X$ is defined by
\begin{eqnarray*}
f x  = y, \qquad \text{with }\ y_i =(f_i x_{A(i)}) x_i, \quad
\mbox{for each }i\in I.
\end{eqnarray*}
\end{defi}
It is easy to verify that this is a faithful action of $F_I$ on
$X$. Therefore $(F_I,X)$ is a permutation group, called the {\it
generalized wreath product of the permutation groups
$(G_i,X_i)_{i\in I}$}.

\begin{defi}
An automorphism of a poset block structure
$(X,\sim_{\mathcal{A}})$ is a permutation $\sigma$ of $X$ such
that, for every equivalence relation $\sim_J$ in
$\sim_{\mathcal{A}}$,
$$
x \sim_J y \qquad \Longleftrightarrow \qquad (\sigma x)\sim_J
(\sigma y), \qquad \mbox{for all } x, y \in X.
$$
\end{defi}
Denote by $Sym(X_i)$ the symmetric group of $X_i$. If $|X_i|=q_i$,
we will also write $Sym(q_i)$. The following fundamental results
are proven in \cite{bayleygeneralized}.
\begin{thm}\label{thmgwp}
The generalized wreath product of the permutation groups $(G_i,
X_i)_{i\in I}$ is transitive on $X$ if and only if $(G_i, X_i)$ is
transitive for each $i\in I$. If $(X, \sim_{\mathcal{A}})$ is the
poset block structure associated with the poset $(I,\preceq)$ and
$F_I$ is the generalized wreath product
$\prod_{(I,\preceq)}Sym(X_i)$, then $F_I$ is the automorphism
group of $(X,\sim_{\mathcal{A}})$.
\end{thm}

\begin{example}
If $\preceq$ is the identity relation (Fig. \ref{figure13}), then
the generalized wreath product is the permutation direct product
$(G_1,X_1)\times(G_2,X_2)\times \cdots \times(G_n,X_n)$. In this
case, we have $A(i) = \emptyset$, for each $i\in I$, so that an
element $f$ of $F_I$ is given by $f = (f_i)_{i\in I}$, where $f_i$
is a function from a singleton $\{\ast\}$ into $G_i$ and so its
action on $x_i$ does not depend on any other coordinate of $x$.
\begin{figure}[h]
\begin{picture}(300,30)
\put(80,20){$\bullet$}\put(110,20){$\bullet$}\put(140,20){$\bullet$}\put(150,23){\circle*{1}}
\put(160,23){\circle*{1}}\put(170,23){\circle*{1}}\put(180,23){\circle*{1}}\put(190,23){\circle*{1}}
\put(200,23){\circle*{1}}\put(210,23){\circle*{1}}\put(220,20){$\bullet$}
\put(78,8){$1$}\put(110,8){$2$}\put(140,8){$3$}\put(220,8){$n$}
\end{picture}\caption{}\label{figure13}
\end{figure}

If $(I,\preceq)$ is a finite chain (Fig. \ref{figure14}), then the
generalized wreath product is the classical permutation wreath
product $(G_1,X_1)\wr(G_2,X_2)\wr \cdots \wr(G_n,X_n)$. In this
case, we have $A(i) = \{1,2,\ldots,i-1\}$, for each $i\in I$, so
that an element $f\in F_I$ is given by $f = (f_i)_{i\in I}$, with
$$
f_i:X_1 \times \cdots \times X_{i-1}\longrightarrow G_i
$$
In other words, the action of $f$ on $x_i$ depends on its \lq\lq
ancestral\rq\rq coordinates $x_1, \ldots, x_{i-1}$.
\begin{figure}[h]
\begin{picture}(250,110)
\letvertex A=(125,100)\letvertex B=(125,80)\letvertex C=(125,60)
\letvertex D=(125,55)\letvertex E=(125,50)\letvertex F=(125,45)
\letvertex G=(125,40)\letvertex H=(125,35)\letvertex I=(125,15)

\drawvertex(A){$\bullet$}\drawvertex(B){$\bullet$}\drawvertex(C){$\bullet$}\drawvertex(H){$\bullet$}
\drawvertex(I){$\bullet$}
\drawvertex(D){\circle*{1}}\drawvertex(E){\circle*{1}}\drawvertex(F){\circle*{1}}\drawvertex(G){\circle*{1}}
\drawundirectededge(A,B){}\drawundirectededge(B,C){}
\drawundirectededge(H,I){}
\put(129,97){$1$}\put(129,77){$2$}\put(129,57){$3$}\put(130,32){$n-1$}\put(129,13){$n$}
\end{picture}\caption{}\label{figure14}
\end{figure}
\end{example}
\begin{remark}
In \cite{donnosubmitted}, the generalized wreath product of graphs
has been introduced (a different notion with respect to
\cite{erschler}), allowing to get the Cayley graph of a
generalized wreath product of groups from the Cayley graphs of the
single factor groups).
\end{remark}

\section{Lumping the generalized crested product}\label{sectionlumpings}

In this section we present three ways for lumping a generalized
crested product of Markov chains. The first one is obtained by
deleting coordinates indexed by some elements of $I$: this is a
completely general method that can be applied without any
assumption on the Markov chains $P_i$'s. On the other hand, the
two other lumping constructions are realized starting from a
lumping partition of each factor set $X_i$: these are the direct
product of lumping partitions, and the generalized product of
lumping partitions, that reflects the ancestral equivalence
relations defined by the poset.

\subsection{Deleting coordinates}\label{forgetting}

Let $(I,\preceq)$ be a finite poset, with $|I|=n$ and let $X_i$,
$X$, $P_i$, $p_i^0$ be as in Definition
\ref{defgencrestedproduct}. Consider the generalized crested
product of Markov chains $\mathcal{P}$ defined in
\eqref{definizioneiniziale}. Observe that any element of $X$ can
be written as $x=(x_1,\ldots,x_n)$, with $x_i\in X_i$.\\
\indent The construction that we are going to introduce is quite
natural. More precisely, we want to lump the Markov chain
$\mathcal{P}$ by \lq\lq deleting\rq\rq some elements from the
poset $I$. Let $R\subseteq I$ and put $\widetilde{I}= I\setminus
R$. We declare equivalent any two elements $(x_1,\ldots,x_n)$ and
$(y_1,\ldots,y_n)$ of $X$ such that $x_i=y_i$ for each $i\in
\widetilde{I}$. This construction produces a lumping consisting of
$t=\prod_{i\in \widetilde{I}}|X_i|$ parts $L_1,\ldots, L_t$ of the
same cardinality $\prod_{i\in R}|X_i|$.

\begin{thm}\label{forgetisgood}
The function $x\mapsto \mathcal{P}(x,L_s)$ is constant on $L_k$,
for every $k,s=1,\ldots, t$.
\end{thm}
\begin{proof} Let $x=(x_1,\ldots,x_n)\in L_k$. We have
\begin{eqnarray*}
\mathcal{P}(x,L_s) &=& \sum_{y\in L_s}p(x,y)=\sum_{y\in L_s}\sum_{i\in I}p_i^0 p_i(x_i,y_i)\prod_{j\in H(i)}u_j(x_j,y_j)\prod_{j\not\in H[i]}\delta_j(x_j,y_j)\\
 &=& \sum_{i\in I}p_i^0\sum_{y\in L_s} p_i(x_i,y_i)\prod_{j\in H(i)}\frac{1}{|X_j|}\prod_{j\not\in H[i]}\delta_j(x_j,y_j).
\end{eqnarray*}
Fixed an index $i$, we have
\begin{eqnarray*}
& & p_i^0\sum_{y\in L_s} p_i(x_i,y_i)\prod_{j\in H(i)}\frac{1}{|X_j|}\prod_{j\not\in H[i]}\delta_j(x_j,y_j)\\
&=& p_i^0 \prod_{j\in H(i)}\frac{1}{|X_j|}\left(\sum_{y\in L_s}p_i(x_i,y_i) \prod_{j\not\in H[i]}\delta_j(x_j,y_j)\right)\\
&=& p_i^0 \prod_{j\in H(i)}\frac{1}{|X_j|}\left(\sum_{y\in L_s}p_i(x_i,y_i) \prod_{j\not\in (H[i]\cup R)}\delta_j(x_j,y_j)\prod_{j\in R\setminus H[i]}\delta_j(x_j,y_j)\right)\\
&=& p_i^0 \prod_{j\in H(i)}\frac{1}{|X_j|}\prod_{j\not\in
(H[i]\cup R)}\delta_j(x_j,y_j)\left(\sum_{y\in L_s}p_i(x_i,y_i)
\prod_{j\in R\setminus H[i]}\delta_j(x_j,y_j)\right),
\end{eqnarray*}
since if $j\not\in R$ all the elements $y\in L_s$ have the same
$j$-th coordinate. We can distinguish two cases. If $i\in R$, then
$$
\sum_{y\in L_s}p_i(x_i,y_i)\prod_{j\in R\setminus
H[i]}\delta_j(x_j,y_j) = \sum_{y\in \overline{L}_s}p_i(x_i,y_i)=
 \prod_{j\in R\cap H(i)} |X_j|,
$$
with $\overline{L}_s=\{y\in L_s\ |\ y_j=x_j \text{ for all }j\in
R\setminus H[i]\}$; if $i\not\in R$, then
$$
\sum_{y\in L_s}p_i(x_i,y_i)\prod_{j\in R\setminus
H[i]}\delta_j(x_j,y_j)=\sum_{y\in \overline{L}_s}p_i(x_i,y_i)=
p_i(x_i,y_i)\prod_{j\in R\cap H[i]} |X_j|.
$$
In both cases the sum is independent of $x\in L_k$, since any
$x\in L_k$ has the same $i$-th coordinate if $i\not\in R$.
\end{proof}

In what follows, we consider the generalized crested product in
which any Markov operator $P_i$ is the uniform operator $U_i$. The
following proposition explains how the operator $\mathcal{P}$
defined in \eqref{definizioneiniziale} changes after performing
the lumping described in Theorem \ref{forgetisgood}. In practice,
we are removing from the tensor product the operators
corresponding to the indices in $R$.

\begin{prop}\label{propositionR}
Let $(I,\preceq)$ be a finite poset and $R\subseteq I$. Let
$\mathcal{P}$ be the generalized crested product of Markov chains
defined in \eqref{definizioneiniziale}, with $P_i=U_i$, for each
$i\in I$. Denote by $\widetilde{\mathcal{P}}$ the lumped Markov
chain defined by $R$ as in Theorem \ref{forgetisgood}. Then
$$
\widetilde{\mathcal{P}}=\sum_{i\in I}p_i^0\left(\bigotimes_{j\in
H[i]\setminus R}U_j\right)\otimes \left(\bigotimes_{j\not\in
(H[i]\cup R)} I_j\right).
$$
\end{prop}
\begin{proof} By performing successive lumpings, it suffices to prove the statement for $|R|=1$. Let
$R=\{r\}$. Recall that the lumping associated with $R$ is obtained
by identifying the elements $x$ and $y$ such that $x_i=y_i$ for
each $i\neq r$. Let $L_s$ be a part of the corresponding
partition. We have
\begin{eqnarray*}
\mathcal{P}(x,L_s) &=& \sum_{i\in I}p_i^0\sum_{y\in L_s}
\prod_{j\in H[i]}\frac{1}{|X_j|}\prod_{j\not\in
H[i]}\delta_j(x_j,y_j).
\end{eqnarray*}
Fixed an index $i$, we can distinguish two cases. If $r\in H[i]$,
we get
\begin{eqnarray*}
p_i^0\sum_{y\in L_s} \prod_{j\in
H[i]}\frac{1}{|X_j|}\prod_{j\not\in H[i]}\delta_j(x_j,y_j)&=&
p_i^0 \left(\sum_{y_r\in X_r} \frac{1}{|X_r|}\right) \prod_{j\in H[i]\setminus \{r\}}\frac{1}{|X_j|}\prod_{j\not\in H[i]}\delta_j(x_j,y_j)\\
&=& p_i^0 \prod_{j\in
H[i]\setminus\{r\}}\frac{1}{|X_j|}\prod_{j\not\in
H[i]}\delta_j(x_j,y_j),
\end{eqnarray*}
where, for $j\not\in H[i]$, $y_j$ is the $j$-th coordinate of any
element $y\in L_s$. If $r\not\in H[i]$, we get
\begin{eqnarray*}
 p_i^0\sum_{y\in L_s} \prod_{j\in H[i]}\frac{1}{|X_j|}\prod_{j\not\in H[i]}\delta_j(x_j,y_j)&=&
 p_i^0  \prod_{j\in H[i]} \frac{1}{|X_j|}\prod_{j\not\in (H[i]\cup \{r\})}\delta_j(x_j,y_j)\left(\sum_{y_r\in X_r}
 \delta_r(x_r,y_r)\right)\\
&=& p_i^0 \prod_{j\in H[i]}\frac{1}{|X_j|}\prod_{j\not\in
(H[i]\cup \{r\})}\delta_j(x_j,y_j),
\end{eqnarray*}
where, for $j\not\in (H[i]\cup \{r\})$, $y_j$ is the $j$-th
coordinate of any element $y\in L_s$. By summing up the terms
corresponding to every $i$, we get the assertion.
\end{proof}
A natural question to be asked is under what conditions, after
performing such a lumping, the lumped Markov chain
$\widetilde{\mathcal{P}}$ represents the generalized crested
product with respect to the reduced poset
$(\widetilde{I},\preceq)$ obtained from $(I,\preceq)$ by deleting
the elements in $R$ and preserving the remaining order relations.
We will use the notation $i\lhd j$ by meaning of $i\prec j$ and
there is no $k\in I$ such that $i\prec k\prec j$.
\begin{defi}
Let $R\subseteq I$ and suppose that, for each $r\in R$, there
exists a unique $s(r) \in H[R]\setminus R$, maximal in
$H[R]\setminus R$, and a sequence $\mathcal{C}(r)=\{r, r',\ldots,
r^{(\ell)}, s(r)\}\subseteq I$, with $r',\ldots, r^{(\ell)}\in R$,
such that $r'\lhd r$, $r^{(k)}\lhd r^{(k-1)}$ and $s(r)\lhd
r^{(\ell)}$. We say that $R$ is {\it fibered} over $S=\cup_{r\in
R}\{s(r)\}$.
\end{defi}
Notice that the chain $\mathcal{C}(r)$ connecting $r$ with $s(r)$
is not necessarily unique. Moreover,
\begin{eqnarray}\label{addedformula}
H(r)\setminus R=\{s(r)\}\sqcup H(s(r))\equiv H[s(r)].
\end{eqnarray}
Roughly speaking, $s(r)$ is the unique maximal element in
$H(r)\setminus R$ that we meet descending from $r$, and there is
no element of $R$ in $H[s(r)]$. Moreover, $s(r)$ is the {\it
maximum} of $H(r)\setminus R$ and among the descendents of $s(r)$
there cannot be elements of $R$ (this is a consequence of the
maximality property of $s(r')$ for all $r'\in R$). If
$S=\{s_1,\ldots, s_t\}$, the {\it fiber} $R_i$ of $s_i$ is the set
of $r\in R$ such that (any) $\mathcal{C}(r)$ ends at $s_i$. If $R$
is fibered over $S$ then $R=\sqcup_{i=1}^t R_i$. In the poset
$(I,\preceq)$ in Fig. \ref{figureRS}, for example, the set
$R=\{2,3,4\}$ is fibered over $S=\{5,6\}$, with fibers
$R_5=\{2,3\}$ and $R_6=\{4\}$; the corresponding reduced poset is
denoted by $(\widetilde{I},\preceq)$. On the other hand, the set
$R'=\{1,2,3,4\}$ is not fibered over any subset $S$, since there
are two maximal elements in $H(1)\setminus R$.
\begin{figure}[h]
\begin{picture}(300,90)
\letvertex A=(70,85)\letvertex B=(20,45)
\letvertex C=(70,45)\letvertex D=(120,45)
\letvertex E=(45,5)\letvertex F=(120,5)

\letvertex AA=(210,70)\letvertex BB=(180,40)
\letvertex CC=(240,40)
\drawvertex(AA){$\bullet$}\drawvertex(BB){$\bullet$}
\drawvertex(CC){$\bullet$} \put(207,75){$1$}\put(176,25){$5$}
\put(236,25){$6$}

\put(-40,50){$(I,\preceq)$}\put(260,50){$(\widetilde{I},\preceq)$}

\drawvertex(A){$\bullet$}\drawvertex(B){$\bullet$}
\drawvertex(C){$\bullet$}\drawvertex(D){$\bullet$}
\drawvertex(E){$\bullet$}\drawvertex(F){$\bullet$}

\drawundirectededge(AA,BB){} \drawundirectededge(AA,CC){}
\drawundirectededge(A,B){} \drawundirectededge(A,C){}
\drawundirectededge(A,D){}\drawundirectededge(E,B){}\drawundirectededge(E,C){}
\drawundirectededge(F,D){}

\put(60,83){$1$}\put(9,43){$2$}
\put(59,43){$3$}\put(124,43){$4$}\put(35,0){$5$}\put(124,0){$6$}
\end{picture}
\caption{}\label{figureRS}
\end{figure}

\begin{prop}\label{deletingreduced}
Let $(I,\preceq)$ and $\mathcal{P}$ be as in
\eqref{definizioneiniziale}, with $P_i=U_i$, for each $i\in I$.
Let $R\subseteq I$ and $\widetilde{\mathcal{P}}$ be the
corresponding lumped Markov chain as in Proposition
\ref{propositionR}. Then $\widetilde{\mathcal{P}}$ is the
generalized crested product associated with the poset
$(\widetilde{I}, \preceq)$ if and only if there exists
$S=\{s_1,\ldots, s_t\}\subseteq H[R]\setminus R$ such that $R$ is
fibered over $S$. If this is the case, one has
$$
\widetilde{p}^0_i=\begin{cases}
p_i^0&\text{if}\  i\not\in S\\
\sum_{j\in R_{s_h}} p^0_j+ p_{s_h}^0&\text{if}\ i=s_h.
\end{cases}
$$
\end{prop}
\begin{proof}
Firstly, suppose that $R$ is fibered over $S$. It follows from
Proposition \ref{propositionR} that $\widetilde{\mathcal{P}}$ is
obtained from $\mathcal{P}$ by forgetting indices in $R$. We have:
\begin{eqnarray*}
\widetilde{\mathcal{P}}\!\!&=&\!\!\sum_{i\in
I}p_i^0\left(\bigotimes_{j\in H[i]\setminus R}U_j\right)\otimes
\left(\bigotimes_{j\not\in
(H[i]\cup R)} I_j\right)\\
\!\!&=&\!\! \sum_{i\in R}p_i^0\!\!\left(\bigotimes_{j\in
H[i]\setminus R}U_j\!\right)\!\!\otimes\!\!
\left(\bigotimes_{j\not\in (H[i]\cup R)} I_j\!\right)
\!+\!\sum_{i\not\in R}p_i^0\!\!\left(\bigotimes_{j\in
H[i]\setminus R}U_j\!\right)\!\!\otimes\!\!
\left(\bigotimes_{j\not\in(H[i]\cup R)} I_j\!\right)\!\!.
\end{eqnarray*}
Since $R$ is fibered over $S$, all indices belonging to the same
fiber $R_m$ give rise to the same operator after deletion, so that
by \eqref{addedformula} we get
\begin{eqnarray*}
\widetilde{\mathcal{P}}\!\!&=&\!\! \sum_{m=1}^t \sum_{i\in R_m}
p_i^0 \!\!\left(\bigotimes_{j\in
H[s_m]}U_j\!\right)\!\!\otimes\!\! \left(\bigotimes_{j\not\in
(H[s_m]\cup R)} I_j\!\right)\!+\! \sum_{m=1}^t
p_{s_m}^0\!\!\left(\bigotimes_{j\in
H[s_m]}U_j\!\right)\!\!\otimes\!\! \left(\bigotimes_{j\not\in
(H[s_m]\cup R)}
I_j\!\right)\\
\!\!&+&\!\!\sum_{i\not\in (R\sqcup
S)}p_i^0\!\!\left(\bigotimes_{j\in H[i]\setminus
R}\!U_j\!\right)\!\!\otimes\!\!
\left(\bigotimes_{j\not\in(H[i]\cup R)}\! I_j\!\right)\\
\!\!&=&\!\! \sum_{m=1}^t \!\!\left(\sum_{i\in R_m} p_i^0
+p_{s_m}^0\!\right)\!\! \cdot \!\!\left(\bigotimes_{j\in
H[s_m]}\!\!U_j\!\right)\!\!\otimes\!\! \left(\bigotimes_{j\not\in
(H[s_m]\cup R)}\!\! I_j\!\right)\\
\!\!&+&\!\!\sum_{i\not\in (R\sqcup
S)}\!p_i^0\!\!\left(\bigotimes_{j\in H[i]\setminus
R}\!U_j\!\right)\!\!\otimes\!\!
\left(\bigotimes_{j\not\in(H[i]\cup R)} \!I_j\!\right)\\
\!\!&=& \!\! \sum_{i\in \widetilde{I}}\widetilde{p}_i\ \!\!
^0\left(\bigotimes_{j\in H[i]\setminus R}\!U_j\right)\otimes
\left(\bigotimes_{j\not\in (H[i]\cup R)} \!I_j\right).
\end{eqnarray*}
On the other hand if, given $R$, there is no $S$ such that $R$ is
fibered over $S$, then only two possibilities may occur: there can
exist $r\in R$ such that $H(r)=\emptyset$; or there is $r\in R$
with at least two maximal elements $s, s'\in H(r)\setminus R$. In
the first case, notice that $\{j\in H[r]\setminus R\}=\emptyset$
and so there is at least a summand in $\widetilde{\mathcal{P}}$
equal to $p^0_r \left(\otimes_{j\not\in R} I_j\right)$ which does
not appear in any generalized crested product. In the second case,
we get a summand containing a tensor product of the form
$$
(U_s\otimes U_{s'})\otimes \left(\otimes_{j\in (H(s)\cup H(s'))}
U_j\right)\otimes\left(\otimes_{j\in I\setminus(R\cup H[s]\cup
H[s'])} I_j\right).
$$
Since the elements $s,s'$ are not comparable, such a product does
not appear in any generalized crested product. This completes the
proof.
\end{proof}

\subsection{Direct product of lumpings}\label{subsectiondirect}

In this section, we introduce a more general construction of
lumpings for the generalized crested product, that we call direct
product of lumpings.\\
\indent Let $(I,\preceq)$ be a finite poset and let $P_i$ be a
lumpable Markov chain on $X_i$, for each $i\in I$. Hence, there is
a collection $\mathcal{L}^i=\{L^i_1,\ldots, L_{k_i}^i\}$ of
subsets of $X_i$ such that $\sqcup_{j=1}^{k_i}L_j^i=X_i$ and, for
each $s=1,\ldots, k_i$, the function $ x\mapsto p_i(x,L^i_s)$ is
constant on each part of $\mathcal{L}^i$. Clearly, the cartesian
product of partitions
$$
\mathcal{L}=\prod_{i\in I}\mathcal{L}^i
$$
provides a partition of the cartesian product $X=\prod_{i\in
I}X_i$.
\begin{thm}\label{thm6}
For all parts $L,L'\in \mathcal{L}$, the function $x\mapsto
\mathcal{P}(x,L')$ is constant on $L$.
\end{thm}
\begin{proof} By direct computation,
\begin{eqnarray}\label{contractingformulas}
\mathcal{P}(x,L') &=& \sum_{y\in L'}\sum_{i\in I}p_i^0 p_i(x_i,y_i)\prod_{j\in H(i)}u_j(x_j,y_j)\prod_{j\not\in H[i]}\delta_j(x_j,y_j)\nonumber\\
 &=& \sum_{i\in I}p_i^0\sum_{y\in L'} p_i(x_i,y_i)\prod_{j\in H(i)}\frac{1}{|X_j|}\prod_{j\not\in H[i]}\delta_j(x_j,y_j).
\end{eqnarray}
For each fixed index $i$, the summand in
\eqref{contractingformulas} becomes
\begin{eqnarray}\label{newstar}
p_i^0 \prod_{j\in H(i)}\frac{1}{|X_j|}\left(\sum_{y\in
L'}p_i(x_i,y_i) \prod_{j\not\in H[i]}\delta_j(x_j,y_j)\right).
\end{eqnarray}
Let $L=\prod_{i=1}^nL^i_{s_i}$ and $L'=\prod_{i=1}^n L^i_{t_i}$,
with $s_i,t_i\in \{1, \ldots, k_i\}$. If there exists an index
$j\not\in H[i]$ such that $s_j\neq t_j$, then \eqref{newstar} is
zero since $\delta_j(x_j,y_j)=0$ for every $y\in L'$. If $s_j=t_j$
for any $j\not\in H[i]$, then we have $\prod_{h\in
H[i]}|L^h_{t_h}|$ elements $y$ such that $y_j=x_j$ for every
$j\not\in H[i]$. Therefore we get
$$
p_i^0\prod_{j\in H(i)}\frac{|L^j_{t_j}|}{|X_j|}\sum_{y_i\in
L^i_{t_i}} p_i(x_i,y_i) = p_i^0\prod_{j\in
H(i)}\frac{|L^j_{t_j}|}{|X_j|}
\widetilde{p}_i(L_{s_i}^i,L^i_{t_i}),
$$
since $P_i$ is lumpable. This proves that $\mathcal{P}(x,L')$ does
not depend on $x$.
\end{proof}

\begin{remark}
This second construction is more general than that one in Section
\ref{forgetting}. In fact, fixed the deletion set $R$, we have
that the corresponding lumping partition is obtained by taking the
universal partition of $X_i$, for $i\in R$, and the identity
partition of $X_i$, for $i\not\in R$.
\end{remark}

\subsection{Generalized product of lumpings}

We define here a more general construction, inspired by the
definition of generalized wreath product of groups given in
\cite{bayleygeneralized} (see Definition \ref{labelbayley}), and
allowing to get new lumping partitions for the generalized crested
product of Markov chains. We are going to define a class of
lumping partitions on $\prod_{i\in I}X_i$ reflecting the
combinatorial structure of the poset $(I,\preceq)$.

Let $\mathbb{L}_i$ be the set of the partitions associated with
all the possible lumpings of the Markov chain $P_i$ on $X_i$, i.e.
$$
\mathbb{L}_i=\{\mathcal{L}^i \ : \text{ the Markov chain }P_i
\text{ on }X_i \text{ is lumpable with respect to
}\mathcal{L}^i\}.
$$
Let $\overline{I}= \{h\in I : A(h)=\emptyset\}$. For any $h\in
\overline{I}$ we fix a lumping partition $\overline{\mathcal{L}}\
\!^{h}\in \mathbb{L}_{h}$. The cartesian partition product
$\prod_{h\in \overline{I} }\overline{\mathcal{L}}\ \!^{h}$ is
clearly a partition of $\prod_{h\in \overline{I}}X_h$.

For any $i\in I$, such that $A(i)\subseteq\overline{I}$, we define
a map
$$
f_i: \prod_{j\in A(i)} \overline{\mathcal{L}}\ \!^j\rightarrow
\mathbb{L}_i,
$$
whose domain is the cartesian product of the lumping partitions
that we have chosen for the sets $X_j, j\in A(i)$, and yielding a
lumping partition of $X_i$. Given an element $L=\prod_{j\in A(i)}
L^j_{h_j} \in \prod_{j\in A(i)} \overline{\mathcal{L}}\ \!^j$ such
that
$$
f_i(L)=\mathcal{L}^i=\{L_1^i,\ldots, L_{r_i}^i\},
$$
then $\coprod_{L}L\times f_i(L)$ is a partition of the set
$\prod_{j\in A(i)}X_j\times X_i $. We denote such a new partition
by $\overline{\mathcal{L}}\ \! ^{A[i]}$. For every $i\in I$ for
which $\overline{\mathcal{L}}\ \! ^{A(i)}$ has been constructed,
we define a map
$$
f_i: \overline{\mathcal{L}}\ \!^{A(i)}\rightarrow \mathbb{L}_i,
$$
providing a new partition $\overline{\mathcal{L}}\ \!^{A[i]}$ of
$\prod_{j\in A[i]}X_{j}$, depending on $f_i$. Continuing this way,
after a finite number of steps we construct a partition of the
whole $X$.
\begin{defi}
The partition of $X=\prod_{i\in I}X_{i}$ induced by the $f_i$' s
is called the generalized product of lumpings.
\end{defi}

\begin{example}
Consider the first poset in Fig. \ref{figureNN}, and suppose that
$X_i=\{0,1\}$ for each $i\in I$. Observe that $A(1) = A(2) =
\emptyset$ and $A(3)=\{1\}$, $A(4)=\{1,2\}$. Choose
$\overline{\mathcal{L}}^1=\{\{0, 1\}\}$ and
$\overline{\mathcal{L}}^2=\{\{0\},\{1\}\}$ and define the maps
$f_3:\overline{\mathcal{L}}^1\to \mathbb{L}_3$ and
$f_4:\overline{\mathcal{L}}^1\times \overline{\mathcal{L}}^2\to
\mathbb{L}_4$ as
$$
f_3(\{0,1\}) =\{\{0\},\{1\}\}\in \mathbb{L}_3
$$
$$
f_4(\{0,1\}\times \{0\}) =\{0, 1\}\in \mathbb{L}_4 \qquad
f_4(\{0,1\}\times \{1\}) =\{\{0\},\{1\}\}\in \mathbb{L}_4.
$$
This construction produces the following partition of $X_1\times
X_2\times X_3\times X_4$:
$$
\{0000,1000,0001,1001\}\sqcup\{0010,1010,0011,1011\}\sqcup\{0100,1100\}\sqcup\{0110,1110\}\sqcup
$$
$$
\sqcup\{0101,1101\}\sqcup\{0111,1111\}.
$$
\indent Consider now the chain in Fig. \ref{figureNN} and still
assume $X_i=\{0,1\}$ for every $i\in I$. Observe that
$A(1)=\emptyset$, $A(2)=\{1\}$, and $A(3)=\{1,2\}$. Choose
$\overline{\mathcal{L}}^1=\{\{0\},\{1\}\}$, and define the map
$f_2:\overline{\mathcal{L}}^1\to \mathbb{L}_2$ as
$f_2(\{0\})=\{\{0\},\{1\}\}\in \mathbb{L}_2$ and
$f_2(\{1\})=\{\{0,1\}\}\in \mathbb{L}_2$. Then we get the
partition $\overline{\mathcal{L}}^{A[2]}=\{\{00\},\{01\},
\{10,11\}\}$ of $X_1\times X_2$. Now define $f_3:
\overline{\mathcal{L}}^{A(3)}\equiv\overline{\mathcal{L}}^{A[2]}\to
\mathbb{L}_3$ as, for instance,
$f_3(\{00\})=f_3(\{10,11\})=\{\{0\},\{1\}\}$, and
$f_3(\{01\})=\{\{0,1\}\}$. With this choice, we obtain the
partition $\{\{000\}, \{001\}, \{010,011\}, \{100,110\},
\{101,111\}\}$ of $X_1\times X_2\times X_3$.
\end{example}

\begin{figure}[h]
\begin{picture}(300,75)

\letvertex A=(230,70)\letvertex B=(230,40)
\letvertex C=(230,10)

\put(225,75){$1$}\put(218,35){$2$}\put(226,-5){$3$}

\drawvertex(A){$\bullet$}\drawvertex(B){$\bullet$}
\drawvertex(C){$\bullet$}

\drawundirectededge(A,B){} \drawundirectededge(B,C){}

\letvertex c=(10,20)\letvertex d=(10,70)

\letvertex e=(70,20)\letvertex f=(70,70)

\drawvertex(c){$\bullet$}\drawvertex(d){$\bullet$}
\drawvertex(e){$\bullet$}\drawvertex(f){$\bullet$}
\drawundirectededge(c,d){} \drawundirectededge(e,d){}
\drawundirectededge(f,e){}

\put(7,6){$3$}\put(7,75){$1$}\put(67,6){$4$}\put(67,75){$2$}
\end{picture}
\caption{}\label{figureNN}
\end{figure}

\begin{thm}\label{thm7}
Any partition of $X$ obtained as a generalized product of lumpings
is a lumping partition of the corresponding generalized crested
product of Markov chains.
\end{thm}

\begin{proof} Any part defined by the construction described above can be
represented as a product
$$
L=L^1_{t_1}\times L^2_{t_2}\times \cdots \times L^n_{t_n}
$$
where $L^i_{t_i}$ is a part of some $\mathcal{L}^i\in
\mathbb{L}_i$. Moreover, for every $i$ such that $A(i)\neq
\emptyset$, the subset $L^i_{t_i}$ belongs to the partition
$\mathcal{L}^i=f_i(\prod_{j\in A(i)}L^j_{t_j})$. Let
$x=(x_1,\ldots,x_n)\in L$ and let $L'$ be another part of the new
partition. Then
$$
\mathcal{P}(x,L')= \sum_{i\in I}p_i^0\sum_{y\in L'}
p_i(x_i,y_i)\prod_{j\in H(i)}\frac{1}{|X_j|}\prod_{j\not\in
H[i]}\delta_j(x_j,y_j).
$$
Fix an index $i$ and consider the summand
$$
p^0_i\prod_{j\in H(i)}\frac{1}{|X_j|}  \sum_{y\in L'}
p_i(x_i,y_i)\prod_{j\not\in H[i]}\delta_j(x_j,y_j).
$$
Notice that $A(i)\subseteq I\setminus H[i]$. Let
$x_{i_1},\ldots,x_{i_m}$ be the entries of $x$ corresponding to
indices in $A(i)$, so that $(x_{i_1},\ldots,x_{i_m})\in
L^{i_1}_{t_{i_1}}\times\cdots\times L^{i_m}_{t_{i_m}}$ and
$f_i(L^{i_1}_{t_{i_1}}\times\cdots\times
L^{i_m}_{t_{i_m}})=\mathcal{L}^i$. We have that, if each of the
$\delta_j(x_j,y_j)$'s is not $0$, then $x_j=y_j$ for every $j\in
A(i)$. Since the $L^j_{t_j}$'s constitute a partition of $X_j$,
this implies that they induce the same partition $\mathcal{L}^i$
under $f_i$, so that $y_i$ belongs to some part $L^i_{s_i}$ of the
same lumping partition $\mathcal{L}^i$ containing $L^i_{t_i}$.
This implies that $\sum_{y\in L'}p_i(x_i,y_i)$ does not depend on
$x_i$, since $\mathcal{L}^i$ is a lumping partition of $X_i$.
\end{proof}

\begin{remark}
Notice that the generalized product of lumpings contains the
direct product defined in Section \ref{subsectiondirect} as a
particular case. In fact, in order to get the partition
$\mathcal{L}=\prod_{i\in I}\mathcal{L}^i$, it is enough to fix
$\overline{\mathcal{L}}^j=\mathcal{L}^j$ for every $j$ such that
$A(j)=\emptyset$, and to define the function $f_i$ to be constant
and equal to $\mathcal{L}^i$ for any other index.
\end{remark}

\section{Insect Markov chain}\label{labelinsect}
In this section we recall some basic properties of the Insect
Markov chain, whose state space is the cartesian product of $n$
finite sets, identified with the boundary of a poset block
structure (see Definition \ref{defiposetblocks}). It was first
introduced in \cite{figatalamanca} in the special case of the
regular rooted tree, then studied in \cite[Chapter 7]{libro} and
\cite{cutoff}, and generalized to different settings in
\cite{crested, orthogonal, generalizedcrested}. We will recall the
construction introduced by the authors in \cite{orthogonal} in the
more general context of orthogonal block structures. Let
$(I,\preceq)$ be a finite poset, with $|I|=n$ and suppose
$X_1=\cdots= X_n=:X$, for every $i\in I$. Let $\mathcal{A}$ be the
set of all ancestral subsets of $(I,\preceq)$. $\mathcal{A}$ has a
natural poset structure: in $(\mathcal{A},\preceq)$, we put
$A_1\preceq A_2$ if and only if $A_1\supseteq A_2$. By the
notation $A_1\lhd A_2$, we mean that $A_1\prec A_2$ and there is
no $A\in \mathcal{A}$ such that $A_1\prec A \prec A_2$. The Markov
operator associated with the Insect Markov chain is defined as
$$
\mathcal{P}_I=\sum_{I\neq A\in \mathcal{A}}
p_A\left(\bigotimes_{j\in A}
I_j\right)\otimes\left(\bigotimes_{j\not\in A} U_j\right),
$$
with
\begin{eqnarray}\label{coefficienti}
p_A = \sum_{\begin{array}{c}\scriptstyle C\subseteq \mathcal{A} \ \text{chain} \\
\scriptstyle C = \{I,A_1,\ldots,A',A\}
\end{array}}\alpha_{I,A_1}\cdots \alpha_{A',A}\left(1-\sum_{A\lhd
L}\alpha_{A,L}\right), \quad \text{for every }A\neq \emptyset,
\end{eqnarray}
and
\begin{eqnarray}\label{coefficienti2}
p_{\emptyset} = \sum_{\begin{array}{c}\scriptstyle C\subseteq \mathcal{A} \ \text{chain} \\
\scriptstyle C = \{I,A_1,\ldots,A',\emptyset\}
\end{array}}\alpha_{I,A_1}\cdots \alpha_{A',\emptyset}.
\end{eqnarray}
Here, for all $A',A\in \mathcal{A}$ such that $A'\lhd A$, the
coefficients $\alpha_{A',A}$ are defined as
\begin{eqnarray}\label{alfa1}
\alpha_{A',A} = \frac{1}{|\{J : J \lhd A'\}|\!\cdot\!|X|+|\{L : A'
\lhd L\}|-|X|\!\cdot\!\sum_{J\lhd A'}\alpha_{J,A'}}.
\end{eqnarray}
In particular, \eqref{alfa1} yields $\alpha_{I,A} = \frac{1}{|\{L
: I \lhd L\}|}$, for each $I\lhd A$. For all $I\lhd A'\lhd A$, in
the case $\alpha_{I,A'} =1$, the coefficient $\alpha_{A',A}$ is
not defined by \eqref{alfa1} but as $\alpha_{A',A} =
\frac{1}{|X|+|\{L : A' \lhd L\}|}$.

The state space of the Insect Markov chain $\mathcal{P}_I$ is the
cartesian product $X^n$, regarded as the boundary of the poset
block structure associated with $(I,\preceq)$. We define
\begin{eqnarray}\label{definitiondistance}
d_I(x,y)= n-\max_{A\in \mathcal{A}}\{|A| : \  x\sim_A y\}, \qquad
\text{for all }x,y\in X^n.
\end{eqnarray}
\begin{lemma}\label{lemmametricspace}
$d_I$ is a distance on $X^n$.
\end{lemma}
\begin{proof}
It is clear that $d_I$ satisfies $d_I(x,y)\geq 0$ for all $x,y\in
X^n$, and $d_I(x,y)= 0$ if and only if $x=y$; moreover, one has
$d_I(x,y)=d_I(y,x)$, for each $x,y\in X^n$. As regard as the
triangular inequality, observe that the condition $d_I(x,z)\leq
d_I(x,y)+d_I(y,z)$ is equivalent to
\begin{eqnarray}\label{disuguaglianzatriangolare}
n+\max_{A\in \mathcal{A}}\{|A| : \  x\sim_A z\}\geq \max_{A\in
\mathcal{A}}\{|A| : \  x\sim_A y\}+\max_{A\in \mathcal{A}}\{|A| :
\ y\sim_A z\}.
\end{eqnarray}
If $\max_{A\in \mathcal{A}}\{|A| : \  x\sim_A y\}+\max_{A\in
\mathcal{A}}\{|A| : \  y\sim_A z\}\leq n$, there is nothing to
prove. If $\max_{A\in \mathcal{A}}\{|A| : \  x\sim_A
y\}+\max_{A\in \mathcal{A}}\{|A| : \  y\sim_A z\}= n+h$, with $h$
a positive integer, then there exist two ancestral sets $A_1$ and
$A_2$ such that $\max_{A\in\mathcal{A}}\{|A| : \  x\sim_A
y\}=|A_1|$ and $\max_{A\in \mathcal{A}}\{|A| : \  y\sim_A
z\}=|A_2|$ and $|A_1|+|A_2|=n+h$. On the other hand $|A_1|+|A_2|-
|A_1\cap A_2|\leq n$. This implies that $|A_1\cap A_2|\geq h$ and,
since $A_1\cap A_2\in  \mathcal{A}$, one has $\max_{A\in
\mathcal{A}}\{|A| : \  x\sim_A z\}\geq h$ and so
\eqref{disuguaglianzatriangolare} holds.
\end{proof}

The Insect Markov chain is invariant under the action of the
generalized wreath product $F_I$ of the groups $Sym(X)$, which
acts by automorphisms on $X^n$ (see Theorem \ref{thmgwp}). This
implies that the probability of reaching $y$ from $x$ only depends
on the distance $d_I(x,y)$.
\begin{example}
Consider the poset $(I,\preceq)$ in Fig. \ref{carrettorigirato}
and assume that $X=\{0,1\}$.
\begin{figure}[h]
\begin{picture}(300,75)
\letvertex A=(80,30)\letvertex B=(60,50)\letvertex C=(100,50)
\drawvertex(A){$\bullet$}\drawvertex(B){$\bullet$}\drawvertex(C){$\bullet$}
\drawundirectededge(A,B){}\drawundirectededge(A,C){}
\put(55,55){$1$}\put(97,55){$2$}\put(76,17){$3$}

\put(00,45){$(I,\preceq)$}\put(275,45){$(\mathcal{A},\preceq)$}

\letvertex AA=(220,30)\letvertex BB=(200,50)\letvertex CC=(240,50)\letvertex DD=(220,5)\letvertex EE=(220,70)
\drawvertex(AA){$\bullet$}\drawvertex(BB){$\bullet$}\drawvertex(CC){$\bullet$}\drawvertex(DD){$\bullet$}\drawvertex(EE){$\bullet$}
\drawundirectededge(AA,BB){}\drawundirectededge(AA,CC){}
\drawundirectededge(AA,DD){}\drawundirectededge(BB,EE){}\drawundirectededge(CC,EE){}
\put(216,75){$\emptyset$}\put(175,47){$\{1\}$}\put(243,47){$\{2\}$}
\put(182,24){$\{1,2\}$}\put(215,-8){$I$}
\end{picture}\caption{}\label{carrettorigirato}
\end{figure}
Then $\mathcal{A}=\{\emptyset, \{1\}, \{2\}, \{1,2\}, I\}$. The
associated poset $(\mathcal{A},\preceq)$ is represented in Figure
\ref{carrettorigirato}. According with \eqref{definitiondistance},
we have:
\begin{itemize}
\item $d_I(000,000)=0$;
\item $d_I(000,001) = 1$;
\item $d_I(000,010) = d_I(000,011) = d_I(000,100) = d_I(000, 101) = 2$;
\item $d_I(000,110) = d_I(000,111) = 3$.
\end{itemize}
By using formulas \eqref{alfa1}, we get:
$$
\alpha_{I,\{1,2\}}=1, \quad
\alpha_{\{1,2\},\{1\}}=\alpha_{\{1,2\},\{2\}}=\frac{1}{4},\quad
\alpha_{\{1\},\emptyset}=\alpha_{\{2\},\emptyset}=\frac{2}{5}
$$
and then, using \eqref{coefficienti} and \eqref{coefficienti2},
$$
p_{\{1,2\}}=\frac{1}{2}, \quad p_{\{1\}}=p_{\{2\}}=\frac{3}{20},
\quad p_{\emptyset}=\frac{1}{5}.
$$
The associated Markov operator is then
\begin{eqnarray*}
\mathcal{P}_I&=& \frac{1}{2}I\otimes I \otimes U +
\frac{3}{20}I\otimes U \otimes U+ \frac{3}{20}U\otimes I \otimes U
+ \frac{1}{5}U\otimes
U \otimes U\\
&=& \frac{1}{80}\cdot\begin{pmatrix}
  28 & 28 & 5 & 5 & 5 & 5 & 2 & 2 \\
  28 & 28 & 5 & 5 & 5 & 5 & 2 & 2 \\
  5 & 5 & 28 & 28 & 2 & 2 & 5 & 5 \\
  5 & 5 & 28 & 28 & 2 & 2 & 5 & 5 \\
  5 & 5 & 2 & 2 & 28 & 28 & 5 & 5 \\
  5 & 5 & 2 & 2 & 28 & 28 & 5 & 5 \\
  2 & 2 & 5 & 5 & 5 & 5 & 28 & 28 \\
  2 & 2 & 5 & 5 & 5 & 5 & 28 & 28
\end{pmatrix},
\end{eqnarray*}
according with the fact that the probability transition $p(x,y)$
is a function of $d_I(x,y)$.
\end{example}

\begin{example}\label{differentdeletions}\rm
Consider the Insect Markov chain on $\{0,1\}^3$ associated with
the first poset in Fig. \ref{threeposets}, whose Markov operator
is
$$
\mathcal{P}_I=\frac{1}{7} U\otimes U\otimes U+ \frac{4}{21}
I\otimes U\otimes U + \frac{2}{3} I\otimes I\otimes U.
$$
If we apply the deletion construction of Section \ref{forgetting}
with $R=\{1\}$, so that $S=\{2\}$, we get
$$
\mathcal{P}_{I'}=\frac{1}{7} U\otimes U+ \frac{4}{21}U\otimes U +
\frac{2}{3}  I\otimes U=\frac{1}{3}U\otimes U + \frac{2}{3}
I\otimes U,
$$
which is the Insect Markov chain on the reduced poset. On the
other hand, with the choice $R'=\{3\}$, we have
$$
\mathcal{P}_{I''}=\frac{1}{7} U\otimes U+ \frac{4}{21}I\otimes U +
\frac{2}{3}  I\otimes I,
$$
which is not a generalized crested product of Markov chains, since
$R'$ is not fibered on any set, according with Proposition
\ref{deletingreduced}.
\begin{figure}[h]
\begin{picture}(200,75)
\letvertex A=(20,70)\letvertex B=(20,40)
\letvertex C=(20,10)
\letvertex a=(100,70)\letvertex b=(100,40)
\letvertex c=(100,10)
\letvertex d=(180,70)
\letvertex e=(180,40)\letvertex f=(180,10)

\drawvertex(a){$\times$}\drawvertex(b){$\bullet$}
\drawvertex(c){$\bullet$}\drawvertex(d){$\bullet$}
\drawvertex(e){$\bullet$}\drawvertex(f){$\times$}\drawvertex(A){$\bullet$}\drawvertex(B){$\bullet$}
\drawvertex(C){$\bullet$}

\drawundirectededge(b,a){}
\drawundirectededge(c,b){}\drawundirectededge(B,A){}
\drawundirectededge(C,B){} \drawundirectededge(e,d){}
\drawundirectededge(f,e){}

\put(8,7){$3$}\put(8,37){$2$}\put(8,67){$1$}
\put(88,7){$3$}\put(88,37){$2$}\put(88,67){$1$}
\put(168,7){$3$}\put(168,37){$2$}\put(168,67){$1$}
\end{picture}
\caption{}\label{threeposets}
\end{figure}
\end{example}

\section{Group actions and lumpability}\label{sectiongroups}

\subsection{General properties}

Many lumpings are generated by the action of a group on a set.
More precisely, let $P$ be a reversible Markov chain with state
space $X$ and let $G$ be a group acting on $X$ such that
\begin{eqnarray}\label{goodG}
p(gx,gy) = p(x,y), \qquad \text{for all }g\in G, x,y\in X.
\end{eqnarray}
The action of $G$ induces a partition of $X$ given by the orbits.
Denote by $G_x$ the orbit of the element $x\in X$ under the action
of $G$. By defining
\begin{eqnarray}\label{lumpingbygroup}
p_G(G_x,G_y) = p(x,G_y) = \sum_{y'\in G_y}p(x,y'),
\end{eqnarray}
then one gets a lumping of $X$ providing a new Markov chain
defined on the space of the orbits. In fact, it is straightforward
to verify that $p_G(G_x,G_y)$ does not depend on the choice of
$x$, so that the lumped Markov chain $P_G$ is well defined. The
following results can be found in \cite{diaconis}.
\begin{prop}
Let $P$ be a reversible Markov chain on $X$ and $G$ a group
satisfying \eqref{goodG}. Let $P_G$ be the Markov chain defined in
\eqref{lumpingbygroup}. If $\widetilde{f}$ is an eigenfunction of
$P_G$ with eigenvalue $\widetilde{\lambda}$, then
$\widetilde{\lambda}$ is an eigenvalue of $P$ with $G$-invariant
eigenfunction $f$ such that $f(x) = \widetilde{f}(G_x)$, for each
$x\in X$. Conversely, every $G$-invariant eigenfunction appears
uniquely from this construction.
\end{prop}
The following proposition provides a condition for an
eigenfunction of $P$ to project to a nontrivial eigenfunction of
$P_G$.
\begin{prop}\label{propositionorbitzero}
Let $f$ be an eigenfunction of $P$ with eigenvalue $\lambda$, and
put $\overline{f}(x) = \sum_{g\in G}f(g^{-1}x)$. If
$\overline{f}\neq 0$, then $\widetilde{f}$ defined as
$\widetilde{f}(G_x) = \overline{f}(x)$ is an eigenfunction of
$P_G$ with eigenvalue $\lambda$.
\end{prop}
The next proposition relies the spectral analysis of a lumped
Markov chain induced by the action of a group, with the
representation theory of the corresponding group.
\begin{prop}\label{propspectralanalysis}
Let $P$ be a reversible Markov chain on $X$, with a transitive
automorphism group $G$ satisfying \eqref{goodG}. Let
$$
L(X) = \bigoplus_{i=0}^{k}V_i
$$
be the isotypic decomposition of $L(X)$ under the action of $G$,
with $V_i = d_i W_i$, where $W_i$ are irreducible representations
of $G$ pairwise non isomorphic. Suppose $X\cong G/H$. Then, the
Markov chain $P_H$ has $\sum_{i=0}^kd_i$ distinct eigenvalues,
with $d_i$ eigenvalues having multiplicity $\dim W_i$ in the
Markov chain $P$.
\end{prop}

\begin{example}[{\bf Spherical lumping on the rooted tree}]\label{exampleAlfredo}\rm
Consider the case of the Insect Markov chain $\mathcal{P}$
associated with the totally ordered set $(I,\preceq)$ in Fig.
\ref{figure14}. We assume $X:=X_1=\cdots = X_n$, with $|X|=q$, so
that the state space can be identified with the set of finite
words of length $n$ over the alphabet $\{0,1,\ldots, q-1\}$.
$\mathcal{P}$ is invariant with respect to the classical iterated
wreath product
$$
G = \underbrace{Sym(q) \wr \cdots \wr Sym(q)}_{n \text{ times}}.
$$
The associated poset block structure is the rooted $q$-ary tree
$T_{q,n}$ of depth $n$, and the action of $G$ is transitive on
each level of the tree, in particular on its boundary identified
with $X^n$. If we fix the element $x_0 = 0^n$ and consider the
subgroup $H = Stab_{G}(x_0)$, then $X^n$ can be regarded as the
homogeneous space $X^n \cong G/H$. It is known
\cite{appendiceharpe, libro, adm, cutoff} that $(G,H)$ is a
Gelfand pair, so the decomposition of the space $L(X^n)$ into
irreducible submodules under the action of $G$ is
multiplicity-free:
$$
L(X^n)=\bigoplus_{j=0}^{n}W_j,
$$
where $W_0$ is the trivial representation and, for every
$j=1,\ldots, n$,
$$
W_j = \underbrace{L(X) \otimes \cdots \otimes L(X)}_{(j-1) \text{
times}}\otimes L(X)_1\otimes \underbrace{L(X)_0\otimes \cdots
\otimes L(X)_0}_{(n-j) \text{ times}},
$$
where $L(X)_0$ denotes the subspace of constant functions in
$L(X)$ and $L(X)_1 = \{f:X\to \mathbb{C}\ |\ \sum_{x\in
X}f(x)=0\}$. Therefore, we have $\dim W_0=1$ and $\dim
W_j=q^{j-1}(q-1)$, for every $j=1,\ldots, n$. $W_0$ is the
eigenspace associated with the eigenvalue $\lambda_0=1$, whereas
the eigenvalue associated with $W_j$, for $j=1,\ldots, n$, is
$\lambda_j= 1-\frac{q-1}{q^{n-j+1}-1}$.\\
\indent It is clear that the orbits in $X^n$ under the action of
$Stab_G(x_0)$ are the spheres $S_r(x_0)$ centered at $x_0$ of
radius $r$, for $r=0,\ldots,n$, defined with respect to the
distance $d_I$ that coincides, in this case where the poset block
structure is the tree $T_{q,n}$, with the usual ultrametric
distance $d$ on the boundary of the tree. Therefore, the lumped
Markov chain $\mathcal{P}_H$ has $n+1$ states, and each of the
eigenvalues $\lambda_j$, for $j=0,\ldots,n$ is an eigenvalue of
$\mathcal{P}_H$ with multiplicity $1$. Up to normalization, the
eigenfunction of $\mathcal{P}_H$ associated with $\lambda_0=1$ is
the constant function; the eigenfunction associated with
$\lambda_j$, for $j=1,\ldots,n$, is the function $\widetilde{f_j}$
such that
\begin{eqnarray}\label{graffenumbered}
\widetilde{f_j}(S_r(x_0))=
\begin{cases}
1&\text{if}\ r<n-j+1\\
\frac{1}{1-q}&\text{if}\ r=n-j+1\\
0&\text{if}\ r>n-j+1.
\end{cases}
\end{eqnarray}
In the case $q=2$, $n=3$, and $x_0=000$, the partition of
$\{0,1\}^3$ induced by the action of $H=Stab_G(000)$ is
$$
\{0,1\}^3= S_0(x_0)\sqcup S_1(x_0)\sqcup S_2(x_0)\sqcup S_3(x_0),
$$
with
$$
S_0(x_0)=\{000\}, \ S_1(x_0)=\{001\}, \ S_2(x_0)=\{010,011\}, \
S_3(x_0)=\{100,101,110,111\}.
$$
The transition probability matrices are
$$
\mathcal{P}=\frac{1}{168}\begin{pmatrix}
  67 & 67 & 11 & 11 & 3 & 3 & 3 & 3 \\
  67 & 67 & 11 & 11 & 3 & 3 & 3 & 3 \\
  11 & 11 & 67 & 67 & 3 & 3 & 3 & 3 \\
  11 & 11 & 67 & 67 & 3 & 3 & 3 & 3 \\
  3 & 3 & 3 & 3 & 67 & 67 & 11 & 11 \\
 3 & 3 & 3 & 3 & 67 & 67 & 11 & 11 \\
 3 & 3 & 3 & 3 & 11 & 11 & 67 & 67 \\
 3 & 3 & 3 & 3 & 11 & 11 & 67 & 67
\end{pmatrix}
\quad \mathcal{P}_H =\frac{1}{168}\begin{pmatrix}
  67 & 67 & 22 & 12 \\
  67 & 67 & 22 & 12 \\
  11 & 11 & 134 & 12 \\
  3 & 3 & 6 & 156
\end{pmatrix}.
$$
The eigenspaces of $\mathcal{P}$ are:
\begin{enumerate}
\item $W_0$, with eigenvalue $1$, of dimension $1$, generated by the function
$$
f_0=(1,1,1,1,1,1,1,1);
$$
\item $W_1$, with eigenvalue $6/7$, of dimension $1$, generated by the function
$$
f_1=(1,1,1,1,-1,-1,-1,-1);
$$
\item $W_2$, with eigenvalue $2/3$, of dimension $2$, generated by the functions
$$
f_{2,1}=(1,1,-1,-1,0,0,0,0) \qquad f_{2,2}=(0,0,0,0,1,1,-1,-1);
$$
\item $W_3$, with eigenvalue $0$, of dimension $4$, generated by the functions
$$
f_{3,1}= (1,-1,0,0,0,0,0,0)\qquad f_{3,2}=(0,0,1,-1,0,0,0,0)
$$
$$
f_{3,3}=(0,0,0,0,1,-1,0,0)\qquad f_{3,4}=(0,0,0,0,0,0,1,-1).
$$
\end{enumerate}
All the eigenspaces of $\mathcal{P}_H$ have dimension $1$  and
they are:
\begin{enumerate}
\item $\widetilde{W}_0$, with eigenvalue $1$, generated by the function $\widetilde{f}_0=(1,1,1,1)$;
\item $\widetilde{W}_1$, with eigenvalue $6/7$, generated by the function $\widetilde{f}_1=(1,1,1,-1)$;
\item $\widetilde{W}_2$, with eigenvalue $2/3$, generated by the function $\widetilde{f}_2=(1,1,-1,0)$;
\item $\widetilde{W}_3$, with eigenvalue $0$, generated by the
function $\widetilde{f}_3= (1,-1,0,0)$,
\end{enumerate}
according with Proposition \ref{propositionorbitzero} and
\eqref{graffenumbered}.
\end{example}

\subsection{The case of the Insect Markov chain on the rooted
tree}\label{cutoffcitation}

We restrict now our attention to the Insect Markov chain on the
cartesian product $X^n=X\times \cdots \times X$, with $|X|=q$,
associated with the totally ordered set $(I,\preceq)$ in Fig.
\ref{figure14}. In this case, the poset block structure is a
regular rooted tree of degree $q$, so that $X^n$ can be identified
with the boundary of $T_{q,n}$.

Fix a vertex $x_0\in X^n$. Using \eqref{coefficienti},
\eqref{coefficienti2}, \eqref{alfa1}, (see also \cite{cutoff}), we
obtain that the transition probabilities associated with this
Markov chain are
\begin{eqnarray}\label{dist01}
\qquad p(x_0,x_0) = p(x_0,x) =
q^{-1}(1-\alpha_1)+\sum_{i=2}^nq^{-i}\alpha_1\cdots
\alpha_{i-1}(1-\alpha_i), \quad \mbox{if } d(x_0,x)=1
\end{eqnarray}
and, more generally,
\begin{eqnarray}\label{distj}
p(x_0,x) = \sum_{i=j}^nq^{-i}\alpha_1\cdots
\alpha_{i-1}(1-\alpha_i), \quad \mbox{if } d(x_0,x)=j>1,
\end{eqnarray}
where the coefficients $\alpha_j$'s satisfy the recursive relation
$\alpha_j=\frac{1}{q+1}+\alpha_{j-1}\alpha_j\frac{1}{q+1}$ and are
described by
\begin{eqnarray}\label{alphaexplicit}
\begin{cases}
\alpha_j=\frac{q^j-1}{q^{j+1}-1}&\text{for}\ \ 1\leq j\leq n-1\\
\alpha_0=1 \qquad \alpha_n=0.
\end{cases}
\end{eqnarray}
Observe that the Markov chain is in detailed balance with the
uniform distribution $\pi$ on $X^n$ given by $\pi(x) =
\frac{1}{q^n}$, for each $x \in X^n$.

Now let $\mathcal{L}=\{L_1, \ldots, L_k\}$ be a lumping of the
Insect Markov chain. Take an element $x_0 \in X^n$ and suppose
that $x_0\in L_r$, for some $r\in \{1,\ldots, k\}$. For every $i =
0,1,\ldots, n$, and $s=1,\ldots, k$, define:
\begin{eqnarray}\label{lambdadef}
\lambda_{i,s} = |\{x\in X^n \ : \ d(x_0,x)=i \text{ and } x\in
L_s\}|.
\end{eqnarray}
In other words, $\lambda_{i,s}$ is the cardinality of the
intersection of the sphere of radius $i$ centered at $x_0$, with
the part $L_s$ of $\mathcal{L}$. In particular, one has:
$$
\lambda_{0,s} = \begin{cases}
1&\text{if}\ s=r\\
0&\text{otherwise}.
\end{cases}
$$
\begin{thm}\label{teoremone}
Let $\mathcal{L}=\{L_1, \ldots, L_k\}$ be a lumping of the Insect
Markov chain on $X^n$. Let $x_0,y_0 \in L_r$ and let
$\lambda_{i,s},\mu_{i,s}$ be the associated coefficients,
respectively, defined as in \eqref{lambdadef}. Then
$$
\lambda_{i,s}=\mu_{i,s} \qquad \text{for each }i=0,\ldots, n
\text{ and } s=1,\ldots,k.
$$
\end{thm}

\begin{proof}
Since $\mathcal{L}$ is a lumping and $x_0,y_0$ belong to the same
part $L_r$, we have $p(x_0,L_s)=p(y_0,L_s)$ for each $s$, and so
$$
\sum_{i=0}^n \left(\sum_{x_s\in L_s,\
d(x_0,x_s)=i}p(x_0,x_s)\right) = \sum_{i=0}^n \left(\sum_{y_s\in
L_s,\ d(y_0,y_s)=i}p(y_0,y_s)\right).
$$
Since the transition probability $p(x_0,x_s)$ (resp. $p(y_0,y_s)$)
only depends on the distance between $x_0$ and $x_s$ (resp. $y_0$
and $y_s$), we then get
\begin{eqnarray}\label{alfredoyoung}
\sum_{i=0}^n \lambda_{i,s} p(x_0,x_s) = \sum_{i=0}^n
\mu_{i,s}p(y_0,y_s).
\end{eqnarray}
Observe now that $ \lambda_{0,s} =\mu_{0,s}= \begin{cases}
1&\text{if}\ s=r\\
0&\text{otherwise}
\end{cases}$. Moreover, it must be
\begin{eqnarray}\label{equalsum}
\sum_{i=0}^n\lambda_{i,s} = \sum_{i=0}^n\mu_{i,s} = |L_s|.
\end{eqnarray}
Hence, by using \eqref{dist01}, \eqref{distj},
\eqref{alphaexplicit} and \eqref{equalsum}, we can rewrite
\eqref{alfredoyoung} as
\begin{eqnarray*}
\sum_{j=1}^{n-1}\left(
(\lambda_{j,s}-\mu_{j,s})\sum_{i=j}^{n-1}\frac{1}{(q^{i+1}-1)(q^i-1)}\right)=0.
\end{eqnarray*}
Suppose now, by the absurd, that $\lambda_{1,s}\neq \mu_{1,s}$. We
can assume, without loss of generality, that
$\lambda_{1,s}-\mu_{1,s}>0$. By dividing, we get
\begin{eqnarray}\label{afterdividing}
\sum_{i=1}^{n-1}\frac{1}{(q^{i+1}-1)(q^i-1)}+\sum_{j=2}^{n-1}\left(
\frac{\lambda_{j,s}-\mu_{j,s}}{\lambda_{1,s}-\mu_{1,s}}\sum_{i=j}^{n-1}\frac{1}{(q^{i+1}-1)(q^i-1)}\right)=0.
\end{eqnarray}
We want to show that the left-hand side of \eqref{afterdividing}
is actually strictly greater than $0$. To see that, observe that
the difference $\lambda_{1,s}-\mu_{1,s}$ must be at least equal to
$1$, whereas the difference $\lambda_{j,s}-\mu_{j,s}$ cannot be
smaller than $-q^{j-1}(q-1)$, for every $j=1,\ldots, n-1$.\\
\indent Therefore, the minimal value that the left-hand side of
\eqref{afterdividing} can take is obtained by replacing the
occurrence $\lambda_{1,s}-\mu_{1,s}$ by $1$ and the occurrences
$\lambda_{j,s}-\mu_{j,s}$ by $-q^{j-1}(q-1)$. After some
computations, one gets
\begin{eqnarray}\label{simplified}
\sum_{j=1}^{n-1}\frac{1-q^j+q}{(q^{j+1}-1)(q^j-1)}.
\end{eqnarray}
By using the decomposition $\frac{1}{(q^{i+1}-1)(q^i-1)} =
-\frac{\frac{q}{q-1}}{q^{i+1}-1}+\frac{\frac{1}{q-1}}{q^i-1}$, one
obtains that the ratio $\frac{1}{q-1}$ is multiplied by
$\frac{1}{q-1}$ in \eqref{simplified}, the ratio $\frac{1}{q^n-1}$
is multiplied by $\frac{-q+q^n-q^2}{q-1}$, and the ratio
$\frac{1}{q^j-1}$ is multiplied by $-(q+1)$, for every
$j=2,\ldots, n-1$. By collecting all terms, we get
$$
\frac{1}{q-1}\left(\frac{1}{q-1}+\frac{q^n-q^2-q}{q^n-1}-(q+1)\sum_{j=1}^{n-2}\frac{1}{c_j}\right),
$$
where, for each $j\geq 1$, we put $c_j =\sum_{k=0}^j q^k$. By
using the inequality $\frac{1}{c_j}<\frac{1}{q^{j-1}(q+1)}$ and
developing the sum, we finally get
$$
\frac{1}{q-1}\left(\frac{1}{q-1}+\frac{q^n-q^2-q}{q^n-1}-(q+1)\sum_{j=1}^{n-2}\frac{1}{c_j}\right)
> \frac{q^{n-3}(2q-1)-1}{q^{n-3}(q-1)^2(q^n-1)}>0
$$
for every $n \geq 2$ (the case $n=1$ is trivial). This is absurd.

The general case can be treated analogously. So let $k$ be the
smallest index such that $\lambda_{k,s}\neq \mu_{k,s}$ and put $S
= \sum_{i=0}^{k-1}\lambda_{i,s}= \sum_{i=0}^{k-1}\mu_{i,s}$. It is
easy to check that it must be
$$
\lambda_{k,s} = hS, \qquad \mu_{k,s} = \overline{h}S, \qquad
\text{ for some } h, \overline{h}\in \{0,1,\ldots, q-1\}.
$$
This implies
$$
-(q-1)S\leq \lambda_{k,s}-\mu_{k,s}\leq (q-1)S.
$$
By iterating this argument, one obtains that, for every $\ell
=1,\ldots, n-k-1$, one has
$$
-q^{\ell}(q-1)S\leq \lambda_{k+\ell,s}-\mu_{k+\ell,s}\leq
q^{\ell}(q-1)S.
$$
On the other hand, if we assume $\lambda_{k,s}-\mu_{k,s}>0$, we
have that $\lambda_{k,s}-\mu_{k,s}$ must be at least equal to $S$.
By arguing as in the previous case, we get the equation
\begin{eqnarray}\label{afterdividing2}
\sum_{i=k}^{n-1}\frac{1}{(q^{i+1}-1)(q^i-1)}+\sum_{j=k+1}^{n-1}\left(
\frac{\lambda_{j,s}-\mu_{j,s}}{\lambda_{k,s}-\mu_{k,s}}\sum_{i=j}^{n-1}\frac{1}{(q^{i+1}-1)(q^i-1)}\right)=0.
\end{eqnarray}
Therefore, the minimal value that the left-hand side of
\eqref{afterdividing2} can take is
\begin{eqnarray*}
\sum_{j=k}^{n-1}\frac{1-q^{j-k+1}+q}{(q^{j+1}-1)(q^j-1)}.
\end{eqnarray*}
This expression can be rewritten as
$$
\frac{1}{q-1}\left(\frac{1}{q^k-1}+\frac{q^{n-k+1}-q^2-q}{q^n-1}-(q+1)\sum_{j=k}^{n-2}\frac{1}{c_j}\right).
$$
By using again the estimate
$\frac{1}{c_j}<\frac{1}{q^{j-1}(q+1)}$, for each $j=k,\ldots,
n-2$, one can show that this is a strictly positive quantity for
every $k\leq n-1$, obtaining the final contradiction.
\end{proof}
\begin{remark}
We have already remarked that $\mathcal{L}$ is a lumping partition
of the Insect Markov chain on $T_{q,n}$ if and only if
\eqref{alfredoyoung} holds. The Theorem \ref{teoremone} ensures
that not only the sums in \eqref{alfredoyoung} coincide, but it
must be $\lambda_{i,s}=\mu_{i,s}$ for all $i,s$.
\end{remark}

In what follows we denote by $z_1,\ldots,z_{q^{n-1}}$ the vertices
of the $(n-1)$-st level of the rooted $q$-ary tree $T_{q,n}$,
identified with the cartesian product $X^{n-1}$. Hence, the
vertices of the $n$-th level whose prefix of length $n-1$ is $z_i$
have the form $z_ix$, with $x\in \{0,1,\ldots,q-1\}$. We write
$T_i$ for the set of vertices of type $\{z_ix\}$. Notice that
$T_i$ is isomorphic to $T_{q,1}$. If
$\mathcal{L}=\{L_1,\ldots,L_k\}$ is a lumping of the Insect Markov
chain on $X^n$, we denote by $T_i^j=T_i\cap L_j$ the set of
vertices in the class $L_j$ belonging to the subtree rooted at the
vertex $z_i\in X^{n-1}$ (see Fig. \ref{tree}). Observe that
$|T^j_i|=\lambda_{0,j}+\lambda_{1,j}$, where $\lambda_{0,j}$ and
$\lambda_{1,j}$ are referred to any element in $T^j_i$.

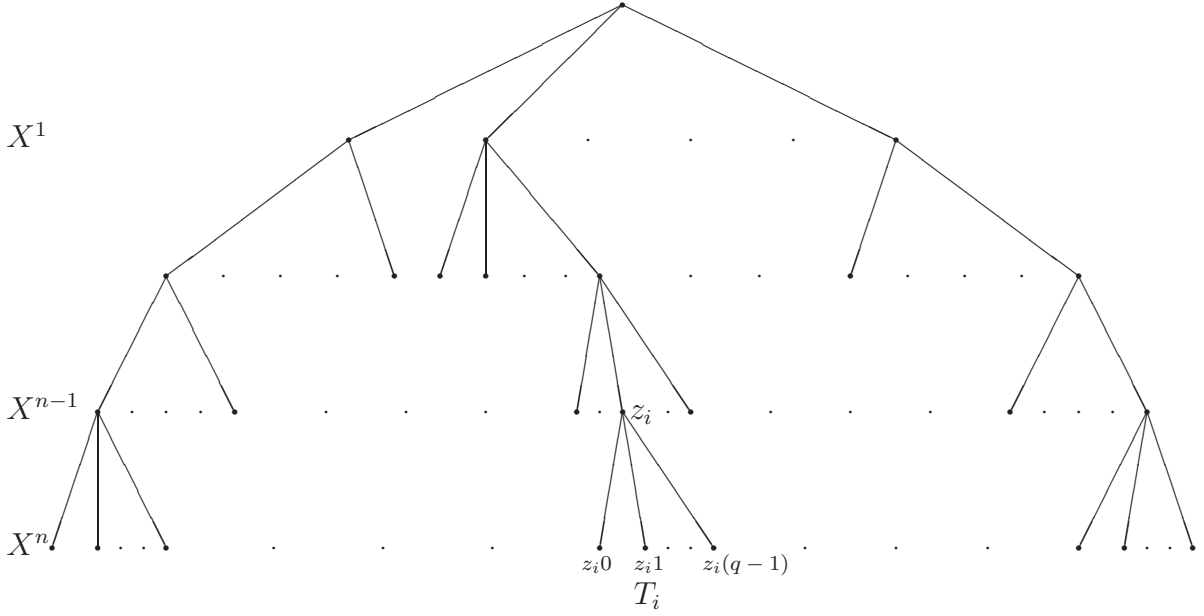
\begin{figure}[h!]
\begin{picture}(500,260)
\put(250,260){\circle*{2}}

\put(130,200){\circle*{2}}\put(190,200){\circle*{2}}
\put(370,200){\circle*{2}}

\put(50,140){\circle*{2}}\put(150,140){\circle*{2}}
\put(170,140){\circle*{2}}\put(190,140){\circle*{2}}\put(240,140){\circle*{2}}
\put(350,140){\circle*{2}}\put(450,140){\circle*{2}}

\put(20,80){\circle*{2}}
\put(80,80){\circle*{2}}\put(230,80){\circle*{2}}\put(250,80){\circle*{2}}
\put(280,80){\circle*{2}}\put(420,80){\circle*{2}}\put(480,80){\circle*{2}}

\put(0,20){\circle*{2}}\put(20,20){\circle*{2}}
\put(50,20){\circle*{2}}\put(240,20){\circle*{2}}\put(260,20){\circle*{2}}
\put(290,20){\circle*{2}}\put(450,20){\circle*{2}}\put(470,20){\circle*{2}}\put(500,20){\circle*{2}}

\put(250,260){\line(-2,-1){120}}\put(250,260){\line(-1,-1){60}}
\put(250,260){\line(2,-1){120}}

\put(130,200){\line(-4,-3){80}}\put(130,200){\line(1,-3){20}}
\put(190,200){\line(-1,-3){20}}
\put(190,200){\line(0,-1){60}}\put(190,200){\line(5,-6){50}}
\put(370,200){\line(-1,-3){20}}\put(370,200){\line(4,-3){80}}

\put(50,140){\line(-1,-2){30}}\put(50,140){\line(1,-2){30}}
\put(240,140){\line(-1,-6){10}}
\put(240,140){\line(1,-6){10}}\put(240,140){\line(2,-3){40}}
\put(450,140){\line(-1,-2){30}}\put(450,140){\line(1,-2){30}}

\put(20,80){\line(-1,-3){20}}\put(20,80){\line(0,-1){60}}
\put(20,80){\line(1,-2){30}}
\put(250,80){\line(-1,-6){10}}\put(250,80){\line(1,-6){10}}
\put(250,80){\line(2,-3){40}}\put(480,80){\line(-1,-2){30}}
\put(480,80){\line(-1,-6){10}}\put(480,80){\line(1,-3){20}}

\put(235,200){\circle*{1}} \put(280,200){\circle*{1}}
\put(325,200){\circle*{1}}

\put(75,140){\circle*{1}} \put(125,140){\circle*{1}}
\put(100,140){\circle*{1}}

\put(207,140){\circle*{1}} \put(225,140){\circle*{1}}

\put(375,140){\circle*{1}} \put(425,140){\circle*{1}}
\put(400,140){\circle*{1}}

\put(35,80){\circle*{1}} \put(65,80){\circle*{1}}
\put(50,80){\circle*{1}} \put(435,80){\circle*{1}}
\put(465,80){\circle*{1}} \put(450,80){\circle*{1}}
\put(240,80){\circle*{1}}\put(270,80){\circle*{1}}

\put(30,20){\circle*{1}} \put(40,20){\circle*{1}}
\put(270,20){\circle*{1}} \put(280,20){\circle*{1}}
\put(480,20){\circle*{1}} \put(490,20){\circle*{1}}

\put(280,140){\circle*{1}} \put(310,140){\circle*{1}}
\put(350,140){\circle*{1}}

\put(120,80){\circle*{1}}\put(155,80){\circle*{1}}
\put(190,80){\circle*{1}}
\put(315,80){\circle*{1}}\put(350,80){\circle*{1}}
\put(385,80){\circle*{1}}

\put(97,20){\circle*{1}}\put(145,20){\circle*{1}}
\put(193,20){\circle*{1}}
\put(330,20){\circle*{1}}\put(370,20){\circle*{1}}
\put(410,20){\circle*{1}}

\put(-20,197){$X^1$}\put(-20,77){$X^{n-1}$}\put(-20,17){$X^n$}

\put(253,77){$z_i$}\put(255,-5){$T_i$}

\tiny{\put(232,10){$z_i0$}\put(255,10){$z_i1$}
\put(285,10){$z_i(q-1)$}}
\end{picture}
\caption{The regular rooted tree $T_{q,n}$}\label{tree}
\end{figure}

\begin{lemma}\label{lemmaL}
Let $\mathcal{L}=\{L_1,\ldots,L_k\}$ be a lumping of the Insect
Markov Chain on $X^n$ and let $L_i'$ be the set of vertices in
$X^{n-1}$ obtained by deleting the last letter from the elements
in $L_i$. If there exist indices $i,j$ such that $L_i'\cap
L_j'\neq \emptyset$, then $L_i'=L_j'$.
\end{lemma}
\begin{proof} If $L_i'\cap L_j'\neq \emptyset$, then there exist two elements $z_hx\in L_i$ and $z_hy\in L_j$.
Any other element in $L_i'$ comes from an element of type
$z_sx'\in L_i$. It follows from Theorem \ref{teoremone} that there
is a bijection between $T_h^j$ and $T^j_s$, so there exists $y'$
such that $z_sy'\in L_j$. Hence, $L_i'=L_j'$.
\end{proof}

The previous lemma implies that the sets $L_i'$'s provide a
partition of $X^{n-1}$. It is enough to choose one representative
for those indices such that $L_i'=L_j'$. Clearly, since $\sqcup_i
L_i=X^n$, we have $\cup_i L_i'=X^{n-1}$. We still denote the
induced partition on $X^{n-1}$ by
$\mathcal{L}'=\{L_1',\ldots,L_h'\}$, where $h\leq k$.

\begin{lemma}\label{lemmalump}
Under the hypotheses of Lemma \ref{lemmaL},
$\mathcal{L}'=\{L_1',\ldots,L_h'\}$ is a lumping for the Insect
Markov chain on $X^{n-1}$.
\end{lemma}
\begin{proof} It follows from Theorem \ref{teoremone} that $\mathcal{L}'$ is a
lumping partition for the Insect Markov chain on $X^{n-1}$ if and
only if, given $z_i, z_j \in L_r'$, the corresponding indices
$\lambda_{k,s}'(z_i)$ and $\lambda_{k,s}'(z_j)$ coincide for every
$k=0,\ldots, n-1$ and $s=1,\ldots, h$. Suppose that there exist
indices $\ell, t$ such that $\lambda_{\ell,t}'(z_i)\neq
\lambda_{\ell,t}'(z_j)$. This implies that there exist $z_ix,
z_jy\in L_r$ such that
$$
\lambda_{\ell+1,t}(z_ix)=|T_i^t|\lambda_{\ell,t}'(z_i)\neq
|T_i^t|\lambda_{\ell,t}'(z_j)=\lambda_{\ell+1,t}(z_jy).
$$
This is a contradiction, since $\mathcal{L}$ is a lumping for the
Insect Markov chain on $X^n$.
\end{proof}

The next result shows that, for the Insect Markov chain associated
with the poset of Fig. \ref{figure14}, any lumping comes from the
action of a suitable automorphism group of the rooted tree
$T_{q,n}$. Denote by $Aut(T_{q,n})=\underbrace{Sym(q) \wr \cdots
\wr Sym(q)}_{n \text{ times}}$ the group of all automorphisms of
the rooted $q$-ary tree of depth $n$.

\begin{thm}\label{teoinsetto}
Let $\mathcal{P}_I$ be the Insect Markov chain on the boundary
$X^n$ of $T_{q,n}$. Let $\mathcal{L}=\{L_1,\ldots, L_k\}$ be a
lumping partition for $\mathcal{P}_I$. Then there exists $K\leq
Aut(T_{q,n})$ such that the orbit partition of $X^n$ under the
action of $K$ is $\sqcup_{i=1}^k L_i$.
\end{thm}
\begin{proof} The proof works by induction on the depth of the tree.

For $n=1$, each part $L_i$ can be represented as $L_i=\{x_{i_1},
\ldots, x_{i_{j_i}}\}\subseteq X$. Then it is enough to take
$K=<\sigma>\leq Sym(q)$, where $\sigma$ is the permutation of
$Sym(q)$ whose cyclic decomposition is given by the product of
cycles $\prod_{i=1}^k(x_{i_1} \cdots x_{i_{j_i}})$.

Let $\mathcal{L}$ be a lumping for the Insect Markov chain on
$X^n$. It follows from Theorem \ref{teoremone} that, if $T_j=
\{z_jx : x\in\{0,1,\ldots,q-1\}\}$ decomposes as
$T_j=T^{i_1}_j\sqcup \ldots \sqcup T^{i_{m_i}}_j$, then for any
other index $s$ such that $T_s^{i_t}\neq \emptyset$ (for
$t=1,\ldots, m_i$), it must be $T_s= T^{i_1}_s\sqcup \ldots \sqcup
T^{i_{m_i}}_s$. Moreover, there exists a bijection between
$T_j^{i_t}$ and $T_s^{i_t}$, for every $t=1,\ldots, m_i$. By
deleting the last letter from vertices in $\mathcal{L}$ we get,
from Lemma \ref{lemmalump}, a lumping partition $\mathcal{L}'$ for
the Insect Markov chain on $X^{n-1}$. By induction, there exists a
subgroup $H \leq Aut(T_{q,n-1})$ whose orbits on $X^{n-1}$ provide
the partition $\mathcal{L}'$.

The subgroup $K$ satisfying the claim is obtained by extending in
a suitable way (compatible with the decomposition in any $T_j$)
the action of elements in $H$ to $Aut(T_{q,n})$, and by adding
elements acting nontrivially only on selected subsets of type
$T_j^{i_t}$. Observe that any $h\in H$ induces a permutation of
$X^{n-1}$. We can associate with $h$ an element $\overline{h}\in
Aut(T_{q,n})$ such that
$$
\overline{h}(z_ix)=\left\{
  \begin{array}{ll}
    z_ix, & \hbox{if $h(z_i)=z_i$} \\
    z_j\sigma_{h,i}(x) & \hbox{if $h(z_i)=z_j$}
  \end{array}
\right.,
$$
 where $\sigma_{h,i}\in Sym(q)$ is a permutation (obtained by product of cycles) such that
$$
\overline{h}(T_i^{i_t})=T_j^{i_t},
$$
for every $t=1,\ldots, m_i$. Moreover, for any non empty set
$T_i^j$ we define the element $g_{i,j}\in Aut(T_{q,n})$ such that
$$
g_{i,j}(z_sx)=\left\{
  \begin{array}{ll}
    z_sx & \hbox{if $s\neq i$} \\
  z_ix & \hbox{if $s=i$ and $x\not\in T_{i}^j$} \\
    z_i \tau_{i,j}(x) & \hbox{if $s=i$ and $x\in T_{i}^j$}
  \end{array}
\right.,
$$
where $\tau_{i,j}\in Sym(q)$ is a cyclic permutation of the
elements in $T_i^j$. Define
$$
K:=<\overline{h}, g_{i,j}\ : \ h\in H, i\in\{1,\ldots, q^{n-1}\},
\ j\in\{1,\ldots, k\}, T^j_i\neq \emptyset>.
$$
By construction, every part of $\mathcal{L}$ is closed under the
action of $K$, since it is closed under the elements
$\overline{h}$ and $g_{i,j}$. The orbit decomposition of $X^n$
under $K$ is given by $L_1\sqcup \ldots \sqcup L_k$, i.e., the
action of $K$ is transitive on each part of $\mathcal{L}$. In
fact, suppose that $z_sx$ and $z_py$ are elements in the same part
$L_r$. Then there exists $h\in H$ such that $h(z_s)=z_p$. By
definition, the automorphism $g=g_{p,r}^c\overline{h}\in K$, where
$c$ satisfies $g_{p,r}^c(\sigma_{h,s}(x))=y$, is such that
$g(z_sx)=z_py$.
\end{proof}

\begin{remark}\rm
The group $K$ is not uniquely determined, i.e., there exist
different subgroups of $Aut(T_{q,n})$ providing to the same
lumping partition.
\end{remark}

In contrast with Theorem \ref{teoinsetto}, for the generalized
Insect Markov chain associated with any poset $(I,\preceq)$ that
is not a chain or, equivalently, that contains two elements which
are not comparable, it is easy to find examples of lumping
partitions that are not induced by any subgroup of the
corresponding generalized wreath product of permutation groups.
\begin{example}\label{nonviene}
Consider the poset $(I,\preceq)$ in Fig. \ref{figure13}, for
$n=2$, and suppose that $X:=X_1=X_2=\{0,1,\ldots, q-1\}$. Let
$\mathcal{P}_I = \frac{1}{2}(U\otimes I+I\otimes U)$ be the Insect
Markov chain associated with $(I,\preceq)$. Take the lumping
$\widehat{\mathcal{L}}$ given by $L_1=\{00, (q-1)(q-1)\}$,
$L_2=\{0(q-1)\}$, $L_3=\{(q-1)0\}$, $L_4=\{0x, (q-1)y : x, y\neq
0, q-1\}$, $L_5=X^2\setminus (\sqcup_{i=1}^4 L_i)$. One can check
that such a partition cannot be induced by any subgroup of
$Sym(q)\times Sym(q)$.
\end{example}
Actually, the following more general result holds.
\begin{prop}\label{prop14}
Let $(I, \preceq)$ be a finite poset, with $|I|=n$, and suppose
that there exist $i,j\in I$ such that $i\nprec j$ and $j \nprec
i$. Let $\mathcal{P}_I$ be the Insect Markov chain on $X^n$
associated with $(I,\preceq)$ and let $F_I$ be the corresponding
automorphism group. Then there exists a lumping of $\mathcal{P}_I$
which is not induced by any subgroup of $F_I$.
\end{prop}
\begin{proof} Let $i$ and $j$ be two elements of $I$ such
that $i\nprec j$ and $j\nprec i$. As in Example \ref{nonviene}, we
can define on the cartesian product $X_i\times X_j\cong X\times X$
a lumping $\widehat{\mathcal{L}}$ which is not induced by any
subgroup of $Sym(X_i)\times Sym(X_j)\simeq Sym(q)\times Sym(q)$.
We can extend such a lumping to a lumping $\mathcal{L}$ of $X^n$
by saying that $x=(x_1, \ldots, x_n)$ and $y=(y_1, \ldots, y_n)$
belong to the same part of $\mathcal{L}$ if and only if
$(x_i,x_j)$ and $(y_i,y_j)$ belong to
the same part of $\widehat{\mathcal{L}}$. We can distinguish two cases.\\
\indent If $A(i)=A(j)=\emptyset$, then the restriction of the
action of $F_I$ on the factors $X_i$ and $X_j$ coincides with the
action of the group $Sym(q)\times Sym(q)$, what implies that we
cannot find any subgroup $H\leq F_I$ whose
orbits coincide with the parts of $\mathcal{L}$.\\
\indent On the other hand, in the case $A(i),A(j)\neq \emptyset$,
we can assume $A:=A(i)=A(j)$. Any element $x\in X^n$ can be
represented as $x=(x_A,x_i,x_j, x_C)$ where $x_A=(x_{l_1},\ldots,
x_{l_t})$, with $l_m\in A$ and $x_C=(x_{c_1},\ldots, x_{c_s})$,
with $\{c_1,\ldots, c_s\}=I\setminus (A\sqcup \{i,j\})$. If we
admit that the lumping partition $\mathcal{L}$ is induced by a
group $K$, then there must exist an $|A|$-tuple $z_A$ and an
automorphism $f=(f_i)_{i\in I}\in K$ such that $f_h(z_A)=\sigma_h
\in Sym(X_h)$, $h=i,j$, with $\sigma_h(0)=q-1$. On the other hand,
the action of $\sigma_h$ on the part of $\mathcal{L}$ induced by
the part $L_2\in \widehat{\mathcal{L}}$, for instance, produces an
element which does not belong to the same part. This is a
contradiction.
\end{proof}

\begin{example}[{\bf Spherical lumping on a poset block structure via Gelfand pairs}]
Take the Insect Markov chain $\mathcal{P}$ associated with a poset
$(I,\preceq)$. Assume that $X:=X_1=\cdots = X_n$, with $|X|=q$. We
know that the state space $X^n$ of $\mathcal{P}$ is the boundary
of the poset block structure associated with $(I,\preceq)$, whose
automorphism group is the generalized wreath product $F_I$ of the
symmetric groups $Sym(X_i)$.
Moreover, $\mathcal{P}$ is invariant with respect to such a group.\\
\indent If we fix the element $x_0 = 0^n$ and consider the
subgroup $H = Stab_{F_I}(x_0)$, then $X^n$ can be regarded as the
homogeneous space $X^n \cong F_I/H$. In \cite{orthogonal} the
authors showed that $(F_I,H)$ is a Gelfand pair, so the
decomposition of the space $L(X^n)$ into irreducible submodules
under the action of $F_I$ is multiplicity-free: $L(X^n) =
\bigoplus_{S\subseteq I \ \text{antichain}}W_S$, with
\begin{eqnarray*}
W_S=\left(\bigotimes_{i\in A(S)}L(X_i)\right)\otimes
\left(\bigotimes_{i\in
S}L(X_i)_1\right)\otimes\left(\bigotimes_{i\not\in
A[S]}L(X_i)_0\right),
\end{eqnarray*}
where the subspaces $L(X_i)_0$ and $L(X_i)_1$ are defined as in
Example \ref{exampleAlfredo}. Observe that the irreducible
submodules are indexed by the antichains of $(I,\preceq)$. On the
other hand, the orbits in $X^n$ under the action of $H$ are also
indexed by the antichains \cite{orthogonal}:
$X^n=\coprod_{S\subseteq I \ \text{antichain}}O_S$, with
$$
O_S= \left(\prod_{i\in H(S)}X_i \right)\times
\left(\prod_{i\not\in H[S]}X_i^0 \right)\times \left(\prod_{i\in
S}X_i^1 \right),
$$
where $X_i^0 = \{0\}$ and $X_i^1= X_i\setminus \{0\}$. Since the
cardinality of the state space of the lumped Markov chain
$\mathcal{P}_H$ is equal to the number of orbits, it follows from
Proposition \ref{propspectralanalysis} that each of the
eigenvalues listed in \cite{orthogonal} for $\mathcal{P}$ is an
eigenvalue of $\mathcal{P}_H$ with multiplicity $1$. The
corresponding eigenvectors can be easily deduced by the analysis
performed in \cite{orthogonal}.
\end{example}



\begin{thebibliography}{99}

\bibitem{cut1} D. Aldous, P. Diaconis, Shuffling cards and stopping
times, Amer. Math. Monthly 93 (1986) 333--348.

\bibitem{aldous} D. Aldous, J. Fill, Reversible Markov Chains and Random Walks on
Graphs, monograph in preparation
(http://www.stat.berkeley.edu/users/aldous/RWG/book.html).

\bibitem{baileycrested} R. A. Bailey, P. J. Cameron, Crested products of
association schemes, J. London Math. Soc. (2) 72 (2005) 1--24.

\bibitem{bayleygeneralized} R. A. Bailey, Cheryl E. Praeger, C. A. Rowley, T. P. Speed, Generalized
wreath products of permutation groups, Proc. London Math. Soc. (3)
47 (1983) 69--82.

\bibitem{appendiceharpe} M. B. Bekka, P. de la Harpe, Irreducibility
of unitary group representations and reproducing kernels Hilbert
spaces. Appendix by the authors in collaboration with R.
Grigorchuk, Expo. Math. 21 (2003) 115--149.

\bibitem{diaconis} S. Boyd, P. Diaconis, P. Parrilo, L. Xiao,
Symmetry analysis of reversible Markov chains, Internet Math. 2
(2005) 31--71.

\bibitem{treesarticle} T. Ceccherini-Silberstein, F. Scarabotti, F.
Tolli, Trees, wreath products and finite Gelfand pairs, Adv. Math.
206 (2006) 503--537.

\bibitem{libro} T. Ceccherini-Silberstein, F. Scarabotti, F.
Tolli, Harmonic Analysis on Finite Groups: Representation theory,
Gelfand pairs and Markov chains. Cambridge Studies in Advanced
Mathematics 108, Cambridge University Press, Cambridge, 2008.

\bibitem{cut2} G.-Y. Chen, L. Saloffe-Coste, The cutoff phenomenon for
ergodic Markov processes, Electron. J. Probab. 13 (2008) 26--78.

\bibitem{addedreference} G. W. Cobb, Y.-P. Chen, An application of Markov chain Monte Carlo to community
ecology, Amer. Math. Monthly 110 (2003) 265--288.

\bibitem{adm} D. D'Angeli, A. Donno, Self-similar groups and finite Gelfand pairs, Algebra Discrete Math. 2 (2007) 54--69.

\bibitem{crested} D. D'Angeli, A. Donno, Crested products of Markov chains, Ann. Appl. Probab. 19 (2009) 414--453.

\bibitem{cutoff} D. D'Angeli, A. Donno, No cut-off phenomenon for the \lq\lq Insect Markov chain\rq\rq, Monatsh. Math. 156 (2009) 201--210.

\bibitem{orthogonal} D. D'Angeli, A. Donno, Markov chains on orthogonal block structures, European J. Combin. 31 (2010) 34--46.

\bibitem{generalizedcrested} D. D'Angeli, A. Donno, Generalized crested products of Markov chains, European J. Combin. 32 (2011) 243--257.

\bibitem{diaconisbook} P. Diaconis, Group representations in probability and statistics, Institute
of Mathematical Statistics Lecture Notes--Monograph Series 11.
Institute of Mathematical Statistics, Hayward, CA, 1988.

\bibitem{cut3} P. Diaconis, The cutoff phenomenon in finite Markov
chains, Proc. Nat. Acad. Sci. U.S.A. 93 (1996) 1659--1664.

\bibitem{cut4} P. Diaconis, The mathematics of mixing things up, J.
Stat. Phys. 144 (2011) 445--458.

\bibitem{donnosubmitted} A. Donno, Generalized wreath products of graphs and
groups, submitted.

\bibitem{erschler} A. Erschler, Generalized wreath products, Int. Math. Res. Not. (2006), Article ID 57835, 14
pp..

\bibitem{farago} A. Farag\'{o}, On the Convergence Rate of Quasi Lumpable Markov
Chains, Proceedings of the Third European Conference on Formal
Methods and Stochastic Models for Performance Evaluation (EPEW'06,
Budapest, Hungary, June 21-22, 2006), 138--147, Springer-Verlag
Berlin, Heidelberg, 2006.

\bibitem{farago2} A. Farag\'{o}, Speeeding Up Markov Chain Monte Carlo
Algorithms, Proceedings of the 2006 International Conference on
Foundations of Computer Science (FCS'06, Las Vegas, Nevada, June
26-29, 2006), 102--108.

\bibitem{figatalamanca} A. Figà-Talamanca, An application of Gelfand
pairs to a problem of diffusion in compact ultrametric spaces, in:
Topics in Probability and Lie Groups: Boundary Theory, CRM Proc.
Lecture Notes 28, Amer. Math. Soc., Providence, RI, 2001, 51--67,

\bibitem{franceschinis} B. Franceschinis, R. R. Muntz, Bounds for
quasi-lumpable Markov chains, Performance Evaluation 20 (1994)
223--243.

\bibitem{symmetrical} F. Scarabotti, F. Tolli, Spectral analysis of finite
Markov chains with spherical symmetries, Adv. in Appl. Math. 38
(2007), 445--481.

\bibitem{plms} F. Scarabotti, F. Tolli, Harmonic analysis on a finite
homogeneous space, Proc. Lond. Math. Soc. (3) 100 (2010) 348--376.

\bibitem{woess} W. Woess, Denumerable Markov chains: Generating
functions, Boundary theory, Random walks on trees. EMS Textbooks
in Mathematics, European Mathematical Society Zuerich, 2009.

\bibitem{zhoulange} H. Zhou, K. Lange, Composition Markov chains of
multinomial type, Adv. in Appl. Probab. 41 (2009), 270--291.

\end{thebibliography}
\end{document}